\documentclass[12pt]{article}
\usepackage[paperwidth = 8.5in, paperheight = 11in, inner = 1in, outer = 1in, top = 1in, bottom = 1in]{geometry}
\usepackage[pdfborder={0 0 0}]{hyperref}
\hypersetup{
    colorlinks=true,
    linkcolor=blue,
}
\usepackage{graphicx}
\usepackage{amsmath, amsthm, amssymb}
\usepackage{tikz,ifthen,calc}
\usepackage{tikz-cd}
\usepackage[all,cmtip]{xy}
\usepackage[english]{babel}
\usepackage[utf8]{inputenc}
\usepackage{mathtools}
\usepackage{array}
\usepackage{color}
\usepackage{url}

\newcommand{\CC}{{\mathbb{C}}}

\newcommand{\FF}{{\mathbb{F}}}
\newcommand{\GG}{{\mathbb{G}}}

\newcommand{\KK}{{\mathbb{K}}}

\newcommand{\NN}{{\mathbb{N}}}

\newcommand{\PP}{{\mathbb{P}}}
\newcommand{\QQ}{{\mathbb{Q}}}
\newcommand{\RR}{{\mathbb{R}}}

\newcommand{\ZZ}{{\mathbb{Z}}}

  \newcommand{\A}{{\mathcal{A}}}
  \newcommand{\B}{{\mathcal{B}}}

  \newcommand{\E}{{\mathcal{E}}}
  \newcommand{\F}{{\mathcal{F}}}
  \newcommand{\G}{{\mathcal{G}}}

\renewcommand{\L}{{\mathcal{L}}}
  
\renewcommand{\O}{{\mathcal{O}}}

  \newcommand{\W}{{\mathcal{W}}}

\renewcommand{\div}{\operatorname{div}}

\newcommand{\sym}{\operatorname{Sym}}

\newcommand{\Hom}{\operatorname{Hom}}

\newcommand{\rank}{\operatorname{rank}}

\newcommand{\Pic}{\operatorname{Pic}}

\newcommand{\spec}{\operatorname{Spec}}

\newcommand{\Sym}{\operatorname{Sym}}

\newcommand{\Qbar}{\overline{\mathbb{Q}}}

\newcommand{\ra}{\rightarrow}

\theoremstyle{plain}
\newtheorem{theorem}{Theorem}[section]
\newtheorem{proposition}[theorem]{Proposition}
\newtheorem{lemma}[theorem]{Lemma}

\newtheorem{corollary}[theorem]{Corollary}

\theoremstyle{definition}
\newtheorem{defn}[theorem]{Definition}

\newtheorem{example}[theorem]{Example}

\newtheorem{remark}[theorem]{Remark}

\newtheorem*{KSConjecture}{Kawaguchi--Silverman Conjecture}

\newif\ifhascomments \hascommentstrue
\ifhascomments
  \newcommand{\sasha}[1]{{\color{red}[[\ensuremath{\spadesuit\spadesuit\spadesuit} #1]]}}
  \newcommand{\brett}[1]{{\color{blue}[[\ensuremath{\heartsuit\heartsuit\heartsuit} #1]]}}
\else
  \newcommand{\sasha}[1]{}
  \newcommand{\brett}[1]{}
\fi

\title{Dynamics of Endomorphisms for \\ Projective Bundles on Elliptic Curves}
\author{Brett Nasserden, Sasha Zotine}
\date{}

\newcommand{\Addresses}{{
  \bigskip
  \footnotesize

  \noindent Brett Nasserden, \textsc{Department of Mathematics \& Statistics,
McMaster University, Hamilton, Ontario, Canada, L8S 4K1}\par\nopagebreak
  \noindent \texttt{E-mail:} \url{nasserdb@mcmaster.ca}

  \medskip

  \noindent Sasha Zotine, \textsc{Department of Mathematics \& Statistics,
McMaster University, Hamilton, Ontario, Canada, L8S 4K1}\par\nopagebreak
  \noindent \texttt{E-mail:} \url{zotinea@mcmaster.ca}
}}

\begin{document}

\maketitle

\begin{abstract}
We study the dynamics of surjective endomorphisms of projective bundles on elliptic curves and relate their dynamical properties to the geometry of the bundle. As an application we prove the Kawaguchi--Silverman conjecture for projective bundles on elliptic curves, thereby completing the conjecture for all projective bundles on curves. Our approach is to use the transition functions of the bundles. This allows us to further prove the conjecture for projective split bundles on a smooth projective variety with finitely generated Mori cone.
\end{abstract}

\section{Introduction}
\label{sec:intro}

Let $X$ be a smooth projective variety defined over a field of characteristic $0$. We study non-isomorphic surjective endomorphisms $f \colon X \ra X$ and their dynamics. The existence of such morphisms often imply that $X$ has some special geometry. For example, the only curves admitting surjective endomorphisms are of genus zero or one; and a classification for surfaces admitting these was carried out by Fujimoto and Nakayama; see \cite{fujimotonoboru2005compact}. In particular, $\PP^1$ bundles on curves of small genus play a prominent role. The purpose of this paper is to expand our knowledge of non-isomorphic endomorphisms of projective bundles on an elliptic curve. 

Our central results show that the dynamics of a morphism of projective bundles on an elliptic curve are determined by the dynamics of a morphism of the base curve. Lesieutre and Satriano show that up to iteration any surjective morphism of a projective bundle sends fibres to fibres \cite[Lemma~6.2]{lesieutresatriano2021ksc} and hence the degree of these maps between fibres provides a well defined invariant of the morphism. We call this the \textit{degree on the fibres}, which in the case of projective bundles, is also equivalent to an invariant called the \textit{relative dynamical degree}; see Definition~\ref{def:relDyndeg}. In the non-split case, we have the following.

\begin{theorem}\label{intro:main1}
Let $X$ be an elliptic curve and $\E$ a degree 0 vector bundle on $X$ that is not a direct sum of line bundles. The projective bundle $\PP(\E)$ does not possess any surjective endomorphism of degree larger than one on the fibres. 
\end{theorem}

This vastly expands the study of projective bundles on elliptic curves in \cite{matsuzawasanoshibata2018surfaces}, which had done the rank two case. Moreover, Theorem \ref{intro:main1} completes the work begun in \cite{lesieutresatriano2021ksc} where it was shown that the dynamics of endomorphisms of projective bundles on elliptic curves can be understood provided that one can understand the dynamics of endomorphisms of degree 0 projective bundles on curves.

When $\E$ is a direct sum of line bundles Theorem~\ref{intro:main1} no longer applies. Indeed, the projectivization of a trivial bundle on an elliptic curve admits surjective morphisms of arbitrary degree on the fibres. We instead see that the endomorphisms are bounded in complexity by the endomorphisms on the base curve. The \textit{dynamical degree} of a morphism makes this more precise: Suppose $f \colon X \rightarrow X$ is a surjective endomorphism. Fixing an ample divisor $H$ on $X$, the \textit{dynamical degree} of $f$ is defined as the limit
\begin{equation*}
	\lambda_1(f) \coloneqq \lim_{n\rightarrow \infty}\bigl((f^n)^*H\cdot H^{\dim X-1}\bigr)^{\frac{1}{n}},
\end{equation*}
which exists and is independent of $H$; see \cite[Remark 9]{kawaguchisilverman2016degrees}. In the 1990s, it was recognized that dynamical degrees are an essential birational invariant of $f$ closely related to its entropy, and they have been intensely studied. References to foundational works in this area include \cite{guedj2005ergodic, dethelinvigny2010entropy, dillerdujardinguedj2010dynamics, favrejonsson2011compact,gromov2003entropy}.

\begin{theorem}\label{intro:main2}
Let $X$ be a smooth projective variety over $\Qbar$ such that its Mori cone is generated by finitely many numerical classes of curves. Fix $\L_0, \L_1,\ldots, \L_r$ to be numerically trivial line bundles on $X$ with $\L_1$ non-torsion and set $\E = \bigoplus_{i=0}^r \L_i$. Suppose that there is a diagram
\begin{equation*}\xymatrix{\PP(\E)\ar[r]^{f}\ar[d]_\pi &\PP(\E)\ar[d]^\pi\\ X\ar[r]_g & X}
\end{equation*}
with $f$ and $g$ surjective. Then the degree of $f$ on the fibres of $\pi$ is at most $\lambda_1(g)$ and $\lambda_1(f)=\lambda_1(g)$.
\end{theorem}

The commutative square in Theorem~\ref{intro:main2} is a mild assumption due to \cite[Lemma~6.2]{lesieutresatriano2021ksc}. As noted above, arbitrary degree on the fibers is possible for torsion bundles (e.g. a trivial bundle). The rank two case for curves is studied in \cite{matsuzawasanoshibata2018surfaces}. However, their methods were not able to treat three or more line bundles. The novelty of our approach compared to theirs is the use of transition functions which allow one to work with arbitrarily many line bundles.

Amerik pioneered the study of surjective endomorphisms of projective bundles in \cite{amerik2003endos}. With Kuznetsova, she showed that if $X=\PP^n$, then $\PP(\E)$ admits a surjective endomorphism of degree larger than one on the fibres if and only $\E$ splits as a direct sum of line bundles \cite[Theorem~3]{amerikkuznetsova2017projbundles}. Similarly, Theorem~\ref{intro:main1} provides evidence for this statement as a more global phenomenon for projective varieties. 

On the other hand, projectivizations of direct sums of are particularly interesting because they are a major source of examples. For example, trivial bundles of arbitrary rank give products with projective space, and more generally direct sums of torsion bundles on elliptic curves provide examples, as do direct sums of line bundles on toric varieties. 

Conceptually, the study of projective bundles is important because one of the terminating steps of the minimal model program is at a \textit{Mori-fibre space}, the simplest examples being projective bundles. This manifests in the dynamical setting as well, for instance Meng and Zhang use the minimal model program to study so-called \textit{int-amplified morphisms} in a series of papers \cite{meng2020intamp, mengzhang2020intamp, mengzhang2022intamp}. They observe that the vast majority
of Mori-fibre spaces which are known to possess non-isomorphic surjective morphisms are projective bundles.

\begin{subsection}{Applications to the Kawaguchi-Silverman conjecture}\label{intro:KSC}
	Our results have applications in arithmetic dynamics. In the last decade an alternative measure of dynamical complexity arising from arithmetic was established by Kawaguchi and Silverman in \cite{kawaguchisilverman2016degrees}. Associated to $H$ is a Weil height function $h_H$ that measures the arithmetic complexity of the point $P\in X(\Qbar)$ relative to $H$. Kawaguchi and Silverman defined the \textit{arithmetic degree} of $P$ with respect to $f$ as
\begin{equation*}
	\alpha_f(P):=\lim_{n\rightarrow \infty}h^+_{H} \bigl( f^n(P) \bigr)^{\frac{1}{n}},
\end{equation*}
when the limit on the right exists and where $h^+_{H} \bigl( f^n(P) \bigr)=\max\{1, h_{H} \bigl( f^n(P) \bigr)\}$. The arithmetic degree measures the growth of heights along a forward orbit of $f$. Kawaguchi and Silverman conjecture a strong relationship between these two distinct notions of complexity. 

\begin{KSConjecture}[{\cite[Conjecture 6]{kawaguchisilverman2016degrees}}]\label{conj:KSC}
	\def\@currentlabel{KS}
	Let $X$ be a normal projective variety defined over $\Qbar$ and let $f\colon X\dashrightarrow X$ be a dominant rational map. Let $P\in X(\Qbar)$ such that $f^n(P)$ is well defined for all $n\geq 1$. If the forward orbit $O_f(P)$ is Zariski dense in $X$ then the limit $\lim_{n\rightarrow\infty}h_H^+(f^n(P))^\frac{1}{n}$ exists and is equal to $\lambda_1(f)$. In other words $\alpha_f(P)=\lambda_1(f)$.
\end{KSConjecture}

Having control over the dynamics of endomorphisms allow us to prove the conjecture for any projective bundle on an elliptic curve. The proof is in three cases: If the bundle is non-split, then it follows immediately from Theorem~\ref{intro:main1} because it is vaccuously true.

\begin{corollary}
	\label{intro:cor1}
	Let $C$ be an elliptic curve defined over $\Qbar$. For any semistable non-split vector bundle $\E$ on $C$ having degree zero, the Kawaguchi--Silverman conjecture is true for any surjective endomorphism of $\PP(\E)$.
\end{corollary}

Otherwise, if the bundle is split, we are able to prove the conjecture using Theorem~\ref{intro:main2} for the non-torsion case, and in the torsion case we dominate the dynamics of our endomorphism using the Albanese variety; see Section~\ref{sec:splitcase} for details.

\begin{corollary}
	\label{intro:cor2}
	Let $X$ be a smooth projective variety over $\Qbar$ such that its Mori cone is generated by finitely many numerical classes of curves and such that the Kawaguchi--Silverman conjecture is true for all surjective endomorphisms of $X$. Fix $\L_0, \L_1,\ldots, \L_r$ to be numerically trivial line bundles on $X$ and set $\E = \bigoplus_{i=0}^r \L_i$. For any surjective endomorphism of $\PP (\E)$, the Kawaguchi--Silverman conjecture holds.
\end{corollary}

In \cite[Theorem 1.3]{matsuzawasanoshibata2018surfaces}, \hyperref[conj:KSC]{KS}~Conjecture was proven for all surjective endomorphisms of projective surfaces. A major part of the proof is the treatment of surfaces of the form $\PP(\L_0\oplus \L_1)$ where $\L_0$ and $\L_1$ are numerically trivial line bundles on a elliptic curve $C$; see \cite[Section 6.2]{matsuzawasanoshibata2018surfaces}. However, their arguments do not generalize to higher rank. The transition function method allows us to recover these results, but without restrictions on the rank of the bundle, and without restriction on the dimension of the base variety. For example, Corollary~\ref{intro:cor2} and \cite[Theorem on p.~998]{bauer1998abelian} proves \hyperref[conj:KSC]{KS}~Conjecture for projective bundles on an abelian variety isogenous to a product of pairwise non-isogenous abelian varieties of Picard number 1.

Finally, Corollaries~\ref{intro:cor1} and \ref{intro:cor2} complete the \hyperref[conj:KSC]{KS}~Conjecture for all projective bundles on smooth curves.

\begin{corollary}
	Let $C$ be a smooth projective curve defined over $\Qbar$. For any vector bundle $\E$ on $C$, the Kawaguchi--Silverman conjecture holds for any surjective endomorphism of $\PP(\E)$. 
\end{corollary}

In general, the \hyperref[conj:KSC]{KS}~Conjecture concerns \textit{dominant rational maps} as opposed to endomorphisms, such as in our results. However, there are substantial challenges to working with rational maps, particularly since computing $\lambda_1(f)$ or $\alpha_f(P)$ for $f$ a dominant rational map is very difficult. For example, in \cite[Main~Theorem]{belldillerjonsson2020transcendental}, Bell, Diller, and Jonsson construct $f \colon \PP^2 \dashrightarrow \PP^2$ of the form
\begin{align*}
	f = g \circ h, & & g(y_1, y_2) = \biggl(-y_1 \frac{1 - y_1 + y_2}{1 - y_1 - y_2}, -y_2 \frac{1 + y_1 - y_2}{1 - y_1 - y_2} \biggr), & & h(y_1, y_2) = (y_1^a y_2^b, y_1^{-b} y_2^a),
\end{align*}
where $a,b \in \ZZ$ are such that $(a+bi)^n \not\in \RR$ for all $n \geq 1$. This rational map $f$ is such that $\lambda_1(f)$ is transcendental. See \cite[Theorem~1.1]{belldillerjonssonkrieger2024transcendental} for a similar result on $\PP^n$. Consequently, the Kawaguchi--Silverman conjecture is completely open for dominant rational maps $\PP^n\dashrightarrow \PP^n$ except in some special cases such as \emph{monomial} or \emph{regular affine automorphisms}; see \cite[Theorem~2(d), Theorem~3]{kawaguchisilverman2014examples}, \cite[Theorem~4]{gregoryasufuku2011monomial}, and \cite[Theorem~A]{lin2019monomial}.

For these reasons, much of the current literature only verifies the \hyperref[conj:KSC]{KS}~Conjecture for surjective endomorphisms. For instance, the conjecture has been verified for surjective endomorphisms of smooth projective surfaces in \cite[Theorem~1.3]{matsuzawasanoshibata2018surfaces}; rationally connected varieties and klt projective varieties admitting an int-amplified endomorphism in \cite[Theorem~1.1]{matsuzawayoshikawa2019intamp} and \cite[Main~Theorem]{mengzhong2024ksc} respectively; and most recently for non-isomorphic endomorphisms of smooth projective threefolds in \cite[Theorem 1.1]{mengzhang2023threefold}. Each of these papers made use of the minimal model program (MMP) in order to reduce the problem to a simpler type of variety, where the conjecture can be tackled directly. Notably, projective bundles form one of the possible endpoints of the MMP, which must be handled directly. From this perspective, our results expand the possible applications of the MMP to the dynamics of varieties.

Other examples where the \hyperref[conj:KSC]{KS}~Conjecture has been verified include Mori dream spaces in \cite[Theorem~4.1]{matsuzawa2020ksc} and hyper--K\"ahler varieties in \cite[Theorem~1.2]{lesieutresatriano2021ksc}. Both of these classes contain varieties with dynamically interesting surjective endomorphisms. For example, toric varieties are all Mori dream spaces and have many endomorphisms; while hyper--K\"ahler varieties may have automorphisms of positive entropy. However, many of the varieties in these classes are much more mysterious, and it is unclear whether Mori dream spaces or hyper--K\"ahler varieties have dynamically interesting morphisms except in small dimensions. In contrast, our transition function method allows us to construct explicit endomorphisms of projective bundles and study their dynamics; see Section~\ref{sec:examples} for some examples. Moreover, we are even able to use this technique to produce explicit dominant rational maps of projective bundles. This allows for the possibility of an explicit study of dynamical systems on projective bundles, analogous to the study of dynamics on projective spaces.	
\end{subsection}

\subsection*{Proof Strategy}

The primary technique we introduce is the transition function method. In spirit this is the projective bundle generalization of the fact that a surjective endomorphism of $\PP^n$ is given by $n+1$ homogeneous polynomials of degree $d$ with no common zero. Surprisingly, this generalization does not seem to appear in the literature. Roughly speaking the transition function method is as follows: Choose a collection of trivializing charts $U_i$ for $\E$. A surjective endomorphism which is degree $d$ on the fibres of the bundle projection $\PP(\E)\rightarrow X$ corresponds to homogeneous degree $d$ polynomials with coefficients in $\O_X(U_i)$ with no common zero, subject to compatibility conditions determined by the transition functions of $\E$. We show in cases of interest that if the degree on the fibres is too large, the compatibility conditions force the polynomials to have a common zero, thereby proving that no morphism of the specified degree on the fibres can exist.

This technique manifests itself in two ways which is reflected in the structure of the paper. First, having explicit transition functions allows us to write down surjective endomorphisms. We apply this in Section~\ref{sec:nonsplit} to prove Theorem~\ref{intro:main1} directly for non-split bundles on elliptic curves. Second, even if we do not have access to explicit transition functions, we may still impose conditions on the existence of surjective endomorphisms. For instance, in Section~\ref{sec:splitcase}, we show that if $X$ and $\E = \bigoplus \L_i$ satisfy the assumptions of Theorem~\ref{intro:main2}, then the existence of surjective endomorphisms is closely tied to the existence of global sections of tensor products of the $\L_i$. In particular, if they are all numerically trivial, this imposes relations on the bundles, and allows us to reduce to the situation where all of the $\L_i$ being torsion.

In \cite{matsuzawasanoshibata2018surfaces} the authors show that that if $\L_0,\L_1$ are torsion line bundles on an elliptic curve $C$ then the dynamics of endomorphisms of $\PP(\L_0\oplus \L_1)$ descend to the dynamics of endomorphisms of $C\times \PP^1$. As the dynamics on this latter variety are well understood they deduce the Kawaguchi-Silverman conjecture. We generalize this in two ways. First we generalize their technique to abelian varieties of arbitrary dimension; see Proposition~\ref{prop:abeliantorsion}. Afterwards, we use the Albanese variety to reduce from torsion line bundles on $X$ to torsion line bundles on an abelian variety, where we apply the generalization of their result; see Proposition~\ref{prop:pullbackReduction} and Corollary~\ref{cor:pullbackRed}. 

\section{Preliminaries}
\label{sec:Prelim}

We discuss some tools for studying the dynamics of projective bundles and how this relates to the \hyperref[conj:KSC]{KS}~Conjecture. Fix $X$ to be a normal projective variety over $\Qbar$ and let $\E$ be a vector bundle on $X$. Central to the study of surjective endomorphisms of projective bundles are commutative diagrams
\begin{equation}
\label{eq:dagger}
\begin{tikzcd}
\PP(\E) \arrow[r, "f"] \arrow[d, "\pi" left] & \PP(\E) \arrow[d, "\pi" right] \\
X \arrow[r, "g"] & X
\end{tikzcd}
\tag{$\dagger$}
\end{equation}
where $f$ and $g$ are surjective morphisms. Following \cite[Tag 01OB]{stacks-project}, we use $\O_{\PP(\E)}(1)$ to denote the canonical quotient of $\pi^* \E$; the restriction of $\O_{\PP(\E)}(1)$ to the fibres of $\pi$ are the bundles $\O_{\PP^{\rank(\E) - 1}}(1)$. If $X$ is a smooth projective variety and $f\colon \PP(\E)\rightarrow \PP(\E)$ is any surjective endomorphism, then there is some integer $n\geq 1$ and a surjective endomorphism $g\colon X\rightarrow X$ such that the diagram
\begin{equation*}
\begin{tikzcd}
\PP(\E) \arrow[r, "f^n"] \arrow[d, "\pi" left] & \PP(\E)\arrow[d, "\pi"] \\ 
X \arrow[r, "g"] & X
\end{tikzcd} 
\end{equation*}
commutes. In fact, such an iterate can be found when one replaces $\PP(\E)$ with any Mori-fibre space on $X$; see \cite[Lemma~6.2]{lesieutresatriano2021ksc} for a proof. Moreover, If $X$ is any normal projective variety then by \cite[Lemma~3.3]{sano2020products} the \hyperref[conj:KSC]{KS}~Conjecture is true for $f$ if and only if the \hyperref[conj:KSC]{KS}~Conjecture is true for $f^n$ for any integer $n\geq 1$. Combining the previous two remarks allows us to replace the morphism $f\colon \PP(\E)\rightarrow \PP(\E)$ with some iterate so that we are always in the situation of a commutative diagram \eqref{eq:dagger}.

This separates the dynamics $f$ into two pieces: the dynamics of $g$ and the dynamics of the induced morphisms on the fibres of $\pi$. The relative dynamical degree makes this precise.

\begin{defn}[{\cite[Definition 2.1]{lesieutresatriano2021ksc}}]
\label{def:relDyndeg}
Let $X$ and $Y$ be normal projective varieties. Suppose that we have a diagram
\begin{equation*}
\begin{tikzcd}
    X \arrow[r, "f"] \arrow[d, "\pi" left] & X \arrow[d, "\pi"] \\ Y \arrow[r, "g"] & Y
\end{tikzcd}
\end{equation*}
where $f,g$, and $\pi$ are surjective morphisms. Fix ample divisors $H$ on $X$ and $W$ on $Y$. We define the \textit{dynamical degree of $f$ relative to $\pi$} by the formula
\begin{equation*}
\lambda_1(f\vert_\pi) \coloneqq \lim_{n\rightarrow \infty}\left((f^n)^*H\cdot (\pi^*W^{\dim Y})\cdot H^{\dim X-\dim Y-1}\right)^{\frac{1}{n}}.
\end{equation*}
\end{defn}
This limit exists and is independent of $H$ and $W$ by \cite[Theorem~1.1]{truong2020relative}. The first dynamical degree is closely connected to the relative dynamical degree. In particular, \cite[Theorem~2.2.2]{lesieutresatriano2021ksc} gives $\lambda_1(f)=\max\{\lambda_1(g),\lambda_1(f\vert_\pi)\}$.
As a consequence, we have a standard approach for attacking the \hyperref[conj:KSC]{KS}~Conjecture.

\begin{corollary}[{\cite[Corollary 3.2]{lesieutresatriano2021ksc}\label{cor:standard1}}]
Suppose that we have a commuting diagram of normal projective varieties defined over $\Qbar$
\begin{equation*}
\begin{tikzcd}
    X \arrow[r, "f"] \arrow[d, "\pi" left] & X \arrow[d, "\pi"] \\ Y \arrow[r, "g"] & Y
\end{tikzcd}
\end{equation*}
where $f,g,\pi$ are all surjective morphisms. When the Kawaguchi--Silverman conjecture holds for $g$ and $\lambda_1(f)=\lambda_1(g)$, the Kawaguchi--Silverman conjecture holds for $f$. 
\end{corollary}	
\begin{proof}
	Let $P\in X(\Qbar)$ be a point with dense orbit under $f$. The image $\pi(P)$ has a dense orbit under $g$ and we have $\alpha_f(P)\geq \alpha_g(\pi(P))=\lambda_1(g)=\lambda_1(f)$ by the Kawaguchi--Silverman conjecture for $g$ and our assumption on the dynamical degree. Moreover, by \cite[Theorem~1.4]{matsuzawa2020bounds} we have $\alpha_f(P)\leq \lambda_1(f)$ and the result follows. 
\end{proof}

When we have a diagram \eqref{eq:dagger}, the fibres of $\pi$ are projective spaces, and the relative dynamical degree can be related to the degree of $f$ restricted to the fibres of $\pi$. 
\begin{proposition}\label{prop:fibreProp}
Let $X$ be a normal projective variety and let $\E$ be a vector bundle on $X$. Suppose that we have a diagram \eqref{eq:dagger}. We have that
\begin{equation*}
f^*(\O_{\PP(\E)}(1))\equiv_{\textnormal{lin}} \O_{\PP(\E)}(\lambda_1(f\vert_\pi))\otimes\pi^*\B
\end{equation*}
for some line bundle $\B$ on $X$.
\end{proposition}

\begin{proof}
Set $\L=\O_{\PP(\E)}(1)$ and fix an ample line bundle $\A$ on $X$. For large enough $N$, we have that $\E \otimes \A^N$ is globally generated, so that $\O_{\PP(\E \otimes \A^N)}(1) = \L \otimes \pi^* \A^N$ is globally generated and in particular nef. Since the Picard group of $\PP(\E)$ is generated by $\L$ and $\Pic(X)$, we must have that $\L \otimes \pi^* \A^{N}\otimes \pi^*\mathcal{C}$ is ample for some nef line bundle $\mathcal{C}$ on $X$, since otherwise $\L + \operatorname{Nef}(X)$ is a full-dimensional cone in $\Pic(\PP(\E))_\QQ$ lying on the boundary of $\operatorname{Nef}(\PP(\E))$. Hence, $W \coloneqq \A^{N}\otimes\mathcal{C}$ and $H \coloneqq \L + \pi^* W$ are ample on $X$, where addition here means in $\Pic(\PP(\E))_\QQ$.

For any nonnegative integer $n$, we may write $(f^n)^*H=\lambda^n\L+\pi^*\B_n$ for some $\B_n \in \Pic(X)$ and $\lambda \in \ZZ$. Following Definition~\ref{def:relDyndeg}, we see
\begin{align*}
    (f^n)^* H \cdot \pi^* W^{\dim X} \cdot H^{\dim \PP(\E) - \dim X - 1} &= (\lambda^n\L+\pi^*\B_n) \cdot \pi^* W^{\dim X} \cdot (\L+\pi^* W)^{r-2} \\
    &= \lambda^n \L^{r-1} \cdot \pi^* W^{\dim X} + \lambda^n \L \cdot \pi^* W^{\dim X + r -2} \\ 
    &~~~ + \L \cdot \pi^* \B_n \cdot \pi^* W^{\dim X} + \pi^* \B_n \cdot \pi^* W^{\dim X + r -2}
\end{align*}
where $r$ is the rank of $\E$. Since the intersection of more than $\dim X$ pullbacks of classes on $X$ is zero, and $\L^{r-2} = \L \cdot \pi^* W = \pi^* W^{\dim X} = 1$, we obtain
\begin{equation*}
(f^n)^* H \cdot \pi^* W^{\dim X} \cdot H^{\dim \PP(\E) - \dim X - 1} = \begin{cases} 2\lambda^n & r = 2 \\ \lambda^n & r > 2 \end{cases}.
\end{equation*}
Taking the limit of the $n$th root of this, we get $\lambda_1(f\vert_\pi) = \lambda$ and the result follows.
\end{proof}

Proposition~\ref{prop:fibreProp} further shows that the relative dynamical degree is an integer when the fibres are projective spaces and $f$ is a morphism. This is further emphasized by the following proposition, which shows that the relative dynamical degree is equal to the degree of $f$ on the fibres of $\pi$.

\begin{proposition}
Let $X$ be a normal projective variety and let $\E$ be a vector bundle on $X$. Suppose that we have a diagram \eqref{eq:dagger}. The relative dynamical degree $\lambda_1(f\vert_\pi)$ is the degree of $f$ on the fibres of $\pi$.
\end{proposition}
\begin{proof}
Set $\L=\O_{\PP(\E)}(1)$ and write $f^*\L=\lambda\L+\pi^*\B$ for some line bundle $\B$ on $X$. Fix $x\in X$ and let $F_x$ be the fibre of $\pi$ above $x$. We have 
\begin{equation*}
\begin{tikzcd}
\PP(\E) \arrow[r,"f"] & \PP(\E) \\ F_x \arrow[u, "i_x"] \arrow[r, "f_x"] & F_{g(x)} \arrow[u, "i_{g(x)}" right]
\end{tikzcd}
\end{equation*}
where the vertical morphisms are closed embeddings and $f_x$ is the restriction of $f$ to the fibre above $x$. If $F$ is any fibre then we have $\L\vert_F=\O_{F}(1)$ and $(\lambda\L+\pi^*\B) \vert_F = \lambda\O_{F}(1)$ as the restriction of $\pi^* \B$ to any fibre is trivial. Therefore, we obtain
\begin{equation*}
f_x^*\O_{F_{g(x)}}(1)=f_x^*i_{g(x)}^*\L=i_x^*f^*\L=i_x^*(\lambda\L+\pi^*\B)=\lambda\O_{F_x}(1)
\end{equation*}
and hence the degree of $f$ on the fibres of $\pi$ (obtained after identifying $F_x$ and $F_{g(x)}$ by an arbitrary isomorphisms with $\PP^{\rank(\E)-1}$) is $\lambda_1(f\vert_\pi)$ by Proposition~\ref{prop:fibreProp}.
\end{proof}    
Consequently, if we are in the situation of \eqref{eq:dagger} and the \hyperref[conj:KSC]{KS}~Conjecture is known for $g$, then in order to prove the \hyperref[conj:KSC]{KS}~Conjecture for $f$, it suffices to show that the degree of $f$ on the fibres of $\pi$ is at most $\lambda_1(g)$.

\begin{remark}
The possible degree on the fibres of $f$ depends on the geometry of $\E$. For example if $\E$ is trivial then $\PP(\E)=\PP^{\rank(\E)-1}\times X$ and any degree on the fibres is possible. 
\end{remark}

\subsection*{The transition function method}

The central strategy of this paper involves using transition functions to explicitly describe surjective endomorphisms of projective bundles. Any surjective morphism $f\colon \PP^n\rightarrow \PP^n$ of dynamical degree $\lambda_1(f)=d$ is determined by $n+1$ polynomials $s_0,s_1\ldots,s_n$ of degree $d$ without a common zero. Working in a more coordinate-free manner, we see that a surjective morphism $\PP^n\rightarrow \PP^n$ with dynamical degree $d$ is given by a surjection $\O_{\PP^n}^{\oplus n+1}\twoheadrightarrow \O_{\PP^n}(d)$. The explicit description of morphisms is essential in computations and the reason why the dynamics of morphisms of projective spaces are much better understood. We generalize this approach to vector bundles by taking into account two additional pieces of information not present for $\PP^n$.
\begin{enumerate}
    \item In the classical situation, $\PP^n=\PP(V)$ where $V$ is a $n+1$ dimensional vector space, and so is a trivial vector bundle on a point. In particular, there are no transition functions in this case. We must take the transition functions of a general projective bundle into account.
    \item For our applications to the Kawaguchi--Silverman conjecture, we need to consider diagrams \eqref{eq:dagger} where the maps $f$ and $g$ are surjective. 
\end{enumerate}
The classical case of morphisms of $\PP^n$ is then obtained by taking $X = \spec \Qbar$.

For a vector bundle $\E$ on $X$ with trivializing open cover $\{U_j\}$, we use $M_{j \leftarrow i} = M_{j,i}$ to denote the transition matrix from $U_i$ to $U_j$, following the standard conventions, e.g. in \cite[Section~1.2]{LePotier1997}. The reason this convention is typically is used is because it interacts nicely with composition, so that the cocycle condition $M_{k,i} = M_{k,j} M_{j,i}$ has indices in a nice order.
\begin{proposition}
\label{prop:compatibilityconditions}
Let $X$ be a normal projective variety over $\overline{\mathbb{Q}}$, $\E$ a vector bundle on $X$ of rank $r$, and suppose that we have a commutative square \eqref{eq:dagger}. Let $\{U_j\}$ be an open cover of $X$ trivializing all of $\E$, $g^* \E$, and $\B$. Specifying a morphism $f$ of degree $d$ on the fibres of $\pi$ is equivalent to giving, for every fixed $j$, an $(r+1)$-tuple of degree $d$ polynomials
\begin{equation*}
    \theta_j \coloneqq (F_{0,j}, F_{1,j}, \ldots, F_{r,j}) \in \O_X(U_j)[t_0, t_1, \ldots, t_r]^{r+1}
\end{equation*}
such that the $F_{i,j}$ do not vanish at a common point in $\PP^r$. Moreover, they must make the associated diagram
\begin{equation}
\label{eq:compatibility}
\begin{tikzcd}
\O_{U_i\cap U_j}^{r+1} \arrow[r, "\theta_i" above, twoheadrightarrow] \arrow[d, "g^* M_{j,i}" left] & \Sym^d \O_{U_i \cap U_j}^{r+1} \arrow[d, "\beta_{i,j} \Sym^d M_{j,i}" right] \\  
\O_{U_i\cap U_j}^{r+1} \arrow[r, "\theta_j", twoheadrightarrow] & \Sym^d \O_{U_i \cap U_j}^{r+1}
\end{tikzcd}
\end{equation}
commute for all $i \neq j$, where $g^* M_{j,i}$ is the transition function for $g^* \E$, $\Sym^d M_{j,i}$ is the transition function for $\sym^d(\E)$, and $\beta_{i,j}$ is the transition function for $\B$.
\end{proposition}

\begin{proof}
By the functorial properties of a projective bundle, the commutative square \eqref{eq:dagger} is equivalent to a surjective morphism of sheaves
\begin{equation*}
\begin{tikzcd}
\theta \colon \pi^*g^*\E \arrow[r, twoheadrightarrow] & f^*\O_{\PP(\E)}(1).
\end{tikzcd}
\end{equation*}
By Proposition~\ref{prop:fibreProp}, this is equivalent to a surjective endomorphism of sheaves
\begin{equation*}
\begin{tikzcd}
\theta \colon \pi^*g^*\E \arrow[r, twoheadrightarrow] & \O_{\PP(\E)}(d) \otimes \pi^*\B.
\end{tikzcd}
\end{equation*}
In other words, $\theta$ is a non-vanishing global section of 
\begin{equation*}
\Hom(\pi^*g^*\E, \O_{\PP(\E)}(d)\otimes \pi^* \B) \cong (\pi^*g^*\E)^\vee \otimes \O_{\PP(\E)}(d)\otimes \pi^* \B.
\end{equation*}
The projection formula establishes that
\begin{equation*}
H^0(\PP(\E), (\pi^*g^*\E)^\vee \otimes \O_{\PP(\E)}(d) \otimes \pi^*\B) = H^0(X,(g^*\E)^\vee \otimes \sym^d \E \otimes \B).
\end{equation*}
This gives an explicit approach to describing a surjective endomorphism. Choose an open covering $\{U_i\}$ of $X$ that trivializes all of $\E$, $g^* \E$, and $\B$. Since the restriction of sheaves is right exact, we have surjective morphisms 
\begin{equation*}
\begin{tikzcd}
\theta\vert_{U_i} \colon g^*\E\vert_{U_i} \arrow[r, twoheadrightarrow] & \bigl( \sym^d \E \otimes \B \bigr) \bigl. \bigr|_{U_i}
\end{tikzcd}
\end{equation*}
We chose $U_i$ so that each of $\E$, $g^* \E$, and $\B$ are trivial so that we obtain surjective morphisms
\begin{equation*}
\begin{tikzcd}
\theta_i \colon \O_{U_i}^{r+1} \arrow[r, twoheadrightarrow] & \Sym^d \O_{U_i}^{r+1}
\end{tikzcd}
\end{equation*}
Each $\Sym^d \O_{U_i}^{r+1}$ can be identified with $\O_X(U_i)[t_0, t_1, \ldots, t_r]$ where the variables $t_j$ are independent of the open $U_i$ (they are the global sections of $\O_{\PP(\E)}(1)$ restricted to $U_i$). The $\theta_i$ are determined by where the basis elements $\textbf{e}_k$ map to, which is a degree $d$ polynomial $F_{i,k} \in \O_X(U_i)[t_0,t_1,\ldots,t_r]$. Moreover, because the $\theta_i$ come from a global morphism on $\PP(\E)$, they cannot have a common zero. The gluing conditions tell us that on each overlap $U_i\cap U_j$ we have a diagram
\begin{equation*}
\begin{tikzcd}
\O_{U_i\cap U_j}^{r+1} \arrow[r, "\theta_i" above, twoheadrightarrow] \arrow[d, "g^* M_{j,i}" left] & \Sym^d \O_{U_i \cap U_j}^{r+1} \arrow[d, "\beta_{i,j} \Sym^d M_{j,i}" right] \\  
\O_{U_i\cap U_j}^{r+1} \arrow[r, "\theta_j", twoheadrightarrow] & \Sym^d \O_{U_i \cap U_j}^{r+1}
\end{tikzcd}
\end{equation*}
where $M_{j,i}$, $g^* M_{j,i}$, and $\beta_{i,j}$ are the transition functions for $\E$, $g^* \E$, and $\B$ respectively on the open covering $\{U_i\}$. Conversely, any collection of surjections $\theta_i$ that satisfy the above compatibility condition glue to a surjection of sheaves as above. 
\end{proof}

\begin{remark}
\label{rem:compatibility}
After identifying $\Sym^d \O_{U_i}^{r+1}$ with $\O_X(U_i)[t_0, t_1, \ldots, t_r]_d$, a degree $d$ polynomial $F$ in $\O_X(U_i)[t_0, t_1, \ldots, t_r]_d$ can be interpreted as a function on the dual bundle $\E^\vee|_{U_i}$. With this interpretation, the transition function of $\Sym^d \E$ from $U_i$ to $U_j$ is the pullback of the transition function of $\E^\vee$ from $U_j$ to $U_i$. The order of the open sets reverses because dualizing reverses the direction of arrows. This gives $\bigl( \sym^d M_{j,i} \bigr) (F) = F \circ M_{j,i}^\text{\sffamily T}$, because the transition function for $\E^\vee$ from $U_j$ to $U_i$ is $M_{j,i}^\text{\sffamily T}$.
\end{remark}

The identity is a surjective endomorphism, which we may verify using the proposition.

\begin{example}
Let $X$ be a normal projective variety over $\Qbar$ and $\E$ a vector bundle on $X$ of rank $r$. Suppose $g$ is the identity, so that $\B$ is trivial and $g^* \E = \E$. Let $\{U_i\}$ be an open cover trivializing $\E$ and, for each $i$, set $\theta_i \coloneqq (t_0, t_1, \ldots, t_r)$. Let us verify that this specifies a surjective endomorphism. None of these vanish at a common point in $\PP^r$ and the commutativity of the diagram in Proposition~\ref{prop:compatibilityconditions} gives us the equality of matrices
\begin{equation*}
    \begin{bmatrix} t_0 & t_1 & \cdots & t_r \end{bmatrix} M_{j,i} = \begin{bmatrix} \bigl( \sym^1 M_{j,i} \bigr) (t_0) & \bigl( \sym^1 M_{j,i} \bigr) (t_1) & \cdots & \bigl( \sym^1 M_{j,i} \bigr) (t_r) \end{bmatrix},
\end{equation*}
which is true by Remark~\ref{rem:compatibility}. Hence this defines a surjective endomorphism.
\end{example}

\section{Bundles on Elliptic Curves}
\label{sec:bundleprelim}

In this section, we explore the properties of bundles on elliptic curves and produce explicit transition functions for these bundles. Let $C$ be a smooth elliptic curve over $\Qbar$, and let $\E$ be any degree zero vector bundle on $C$. We recall some basic facts about bundles on $C$ established in \cite{atiyah1957bundles}.

\begin{theorem}\label{rmk:AtiyahClass}
\begin{enumerate}
    \item[(a)]  For each positive integer $r$, there is a unique indecomposable vector bundle on $C$ having degree zero, rank $r$, and a nonzero global section. We denote this vector bundle by $\F_r$ and refer to it as the Atiyah bundle of rank $r$. Furthermore, we have $h^0(C, \F_r)=1$; see \cite[Theorem~5]{atiyah1957bundles}. 
    
    \item[(b)] Every indecomposable degree zero vector bundle of rank $r$ is of the form $\F_r\otimes \L$ for a unique degree zero line bundle $\L$; see \cite[Theorem~5]{atiyah1957bundles}.
    
    \item[(c)] We have $\F_r\otimes \F_s=\F_{r+s-1} \oplus \F_{r+s-3} \oplus \F_{r+s-5} \oplus \cdots \oplus \F_{|r-s|+1}$; see \cite[Theorem~8]{atiyah1957bundles}.
    
    \item[(d)] We have $\det \F_r=\O_C$; see \cite[Theorem~5]{atiyah1957bundles}.
    
    \item[(e)] Let $\L$ be a line bundle of degree $0$. The vector bundle $\F_r\otimes \L$ has a global section if and only if $\L=\O_C$; see \cite[Lemma~17]{atiyah1957bundles}.  
    
    \item[(f)] The Atiyah bundle $\F_r$ is self dual; see \cite[Corollary 1]{atiyah1957bundles}.
    
    \item[(h)] We have $\F_r=\sym ^{r-1}\F_2$; see \cite[Theorem~9]{atiyah1957bundles}.
\end{enumerate}
\end{theorem}

We present further properties of Atiyah bundles following the exposition in \cite{nasserden2021thesis}.

\begin{proposition}\label{prop:symprop}
Let $C$ be an elliptic curve defined over $\Qbar$. The Atiyah bundle $\F_r$ of rank $r$ on $C$ satisfies
\begin{equation*}
    \sym^d (\F_r) = \bigoplus_i \F_{r_i}
\end{equation*}
for some integers $r_i$.
\end{proposition}

\begin{proof}
We have a canonical surjection 
\begin{equation*}
    \F_r^{\otimes d} \twoheadrightarrow \sym^d (\F_r) \rightarrow 0.
\end{equation*} 
By Theorem~\ref{rmk:AtiyahClass} (c), we have that $\F_r^{\otimes d} = \bigoplus_i \F_{l_i}$ for some integers $l_i$. The Atiyah bundle is semistable \cite[fact on p.~3]{tu1993semistable}, which implies that $\deg \bigl( \sym^d(\F_r) \bigr) \geq 0$. Taking duals, we obtain the exact sequence
\begin{equation*}
    0 \rightarrow (\sym^d \F_r)^* \rightarrow (\F_r^{\otimes d})^*.
\end{equation*}
Since we are in characteristic 0 and $\F_r^* = \F_r$ we have that 
\begin{equation*}
(\sym^d \F_r)^* = \sym^d \F_r^* = \sym^d \F_r
\end{equation*}
and $(\F_r^{\otimes d})^* = \F_r^{\otimes d}$. Thus $\sym^d \F_r$ is a sub-bundle of $\F_r^{\otimes d}$. Moreover, $\F_r^{\otimes d}$ is semistable of degree zero (being the tensor product of semistable vector bundles) and so we obtain $\deg \sym^d \F_r \leq 0$. Hence, we get $\deg \sym^d \F_r = 0$, so $\sym^d \F_r$ is semistable of degree zero. It follows that each of its summands is also semistable of degree zero and so by Theorem~\ref{rmk:AtiyahClass} (b) we have that 
\begin{equation*}
    \sym^d \F_r = \bigoplus_{j} \F_{r_j} \otimes \L_j
\end{equation*}
where each $\L_j$ is some degree zero line bundle. Whenever $\F$ and $\G$ are vector bundles, we get $\mathcal{H}om(\F,\G) \cong \F^\vee \otimes \G$ and $H^0(C,\mathcal{H}om(\F,\G)) = \Hom(\F,\G)$. Applying this in our situation gives
\begin{align}
\label{eq:sympropeq}
    \Hom(\F_r^{\otimes d},\sym^d\F_r)=\Hom \biggl( \bigoplus_{i}\F_{l_i},\bigoplus_j \F_{r_j}\otimes \L_j \biggr) = \bigoplus_{i,j}\Hom(\F_{l_i},\F_{r_j}\otimes \L_j).
\end{align}
Now we have that $\Hom(\F_{l_i},\F_{r_j}\otimes \L_j)=H^0(C,\F_{l_i}^*\otimes \F_{r_j}\otimes \L_j)=H^0(C,\F_{l_i}\otimes \F_{r_j}\otimes \L_j)$ as the Atiyah bundle is self dual (Theorem~\ref{rmk:AtiyahClass} (f)). Suppose that $\L_j\neq \O_C$ for some $j$. As $\L_j$ is of degree $0$ we have that $H^0(C, \F_{l_i}\otimes \F_{r_j}\otimes \L_j)=0$ by Theorem~\ref{rmk:AtiyahClass} (e). Every surjective map $\psi\colon \F_r^{\otimes d}\twoheadrightarrow \sym^d(\F_r)$ arises as an element of 
\begin{equation*}
\Hom(\F_r^{\otimes d},\sym^d(\F_r))=H^0(C,\mathcal{H}om(\F_r^{\otimes d},\sym^d(\F_r)),
\end{equation*}
so that we have a decomposition $\psi=\bigoplus \psi_{i,j}$ with $\psi_{i,j}\colon \F_{l_i}\rightarrow \F_{r_j}\otimes \L_j$. If $\L_{j_0}\neq \O_C$ for some fixed $j_0$ then $\psi_{ij_0}=0$ for all $i$ by \eqref{eq:sympropeq}. This contradicts the assumption that $\psi$ is surjective. This is because locally the image of $\psi=\bigoplus\psi_{i,j}$ must generate $\F_{r_{j_0}}\otimes \L_{j_0}$. If each $\psi_{ij_0}=0$ then this image is always zero locally, while $\F_{r_{j_0}}\otimes \L_{j_0}$ is nonzero locally. So we have that each $\L_j=\O_C$ and the claim follows.
\end{proof}

As a corollary we may extend the above result to direct sums of Atiyah bundles.

\begin{corollary}\label{cor:Atiyahsum}
Let $C$ be an elliptic curve defined over $\Qbar$ and let $\F_r$ be the rank $r$ Atiyah bundle on $C$. Whenever $\E=\bigoplus_{i=1}^s \F_{r_i}$, we get 
\begin{equation*}
    \sym^d(\E)=\bigoplus_{j=1}^N \F_{w_j}
\end{equation*}
for some integers $w_j$.
\end{corollary}
\begin{proof}
We have that
\begin{equation*}
\sym^d\biggl(\bigoplus_{i=1}^s\F_{r_i}\biggr)\cong \bigoplus_{t_1+\cdots+t_s=d}\bigotimes_{j=1}^s \sym^{t_j}(\F_{r_j})
\end{equation*}
where $t_1, t_2, \ldots, t_s$ are nonnegative integers. Proposition~\ref{prop:symprop} gives $\sym^{t_j}(\F_{r_j})=\bigoplus_{k} \F_{g_{j,k}}$ for some integers $g_{j,k}$. It follows that we have
\begin{equation*}
\bigotimes_{j=1}^s \sym^{t_j}(\F_{r_j})=\bigotimes_{j=1}^s\biggl(\bigoplus_{k}\F_{g_{j,k}} \biggr).
\end{equation*}
By Theorem~\ref{rmk:AtiyahClass} (c), the tensor product of two Atiyah bundles is a direct sum of Atiyah bundles, and the result follows.
\end{proof}

\begin{lemma}\label{lemma:Bistrivial}
For any commutative square
\begin{equation}
\label{eq:cdagger}
\begin{tikzcd}
\PP(\E) \arrow[r, "f"] \arrow[d, "\pi" left] & \PP(\E) \arrow[d, "\pi" right] \\
C \arrow[r, "g"] & C
\end{tikzcd}
\tag{$\ddagger$}
\end{equation}
where $\E=\F_{r+1}$, there is an integer $d$ such that $f^*\O_{\PP(\F_{r+1})}(1) \cong \O_{\PP(F_{r+1})}(d)$.
\end{lemma}

\begin{proof}
The vector bundle $\F_{r+1}$ is a nef and non-ample vector bundle, because it is an iterated extension of nef and non-ample vector bundles; see \cite[Theorem 6.2.12]{lazarsfeld2004positivity}. Since $C$ has Picard number one, $\PP (\F_{r+1})$ has Picard number two and its nef cone is generated by $\O_{\PP (\F_{r+1})}(1)$ and $\pi^* H$ where $H$ is an ample line bundle on $C$. By Proposition~\ref{prop:fibreProp}, we have $f^*\O_{\PP(\F_{r+1})}(1)\equiv_{\textnormal{lin}} \O_{\PP(\F_{r+1})}(d)\otimes \pi^* \B$ for some line bundle $\B$ on $C$. Since $\F_{r+1}$ is of degree zero, so is $\pi^* \B$. Moreover, because $\O_{\PP(\F_{r+1})}(1)$ has a nonzero global section, we get that $f^*\O_{\PP(F_{r+1})}$ has a nonzero global section. In other words $H^0(\PP(\F_{r+1}),  \O_{\PP(\F_{r+1})}(d)\otimes \pi^* \B)=H^0(C, \sym^d(\F_{r+1})\otimes \B)$ is non-empty. However, $\sym^d\F_{r+1}$ is a direct sum of vector bundles of the form $\F_{s}$ by Proposition~\ref{prop:symprop}, and hence
\begin{equation*}
    H^0(C, \sym^d(\F_{r+1})\otimes \B)=\bigoplus_{i=1}^tH^0(C, \F_{s_i}\otimes \B) 
\end{equation*}
is nonzero. This occurs only when $\B=\O_C$ as claimed. 
\end{proof}

\begin{proposition}
\label{prop:Atiyahpreserved}
Let $C$ be an elliptic curve defined over $\Qbar$. For any surjective endomorphism $f \colon C \rightarrow C$ and Atiyah bundle $\F_r$, we have $f^*\F_r\cong \F_r$. 
\end{proposition}

\begin{proof}
Any surjective endomorphism of $C$ can be written $f=\tau_c \circ g$ where $g$ is an isogeny and $\tau_c$ is translation by some element of $C$. We have that $\tau_c^*\F_r$ is indecomposable of degree zero. It also has a nonzero section because $\F_r$ does, and the pull back is an isomorphism on sections. Hence by Theorem~\ref{rmk:AtiyahClass}, we have that $\tau_c^*\F_r\cong \F_r$. It follows by \cite[Corollary 2.1]{oda1971bundles} that we have $g^*\F_r\cong \F_r$ as $g$ is an isogeny. Since $f^*\F_r=(\tau_c\circ g)^*\cong g^*\tau_c^* \F_r$ we have the result. 
\end{proof}


Finally, we describe the transition functions of bundles on elliptic curves as in \cite{zotine2017thesis}. Fix an elliptic curve $C$ in Legendre form
\begin{align}
\label{eq:legendre}
    C \mathrel{:} Z Y^2 = X(X - Z)(X - \lambda Z),
\end{align}
for some $\lambda \in \Qbar \backslash\{0,1\}$. There are exactly three non-identity $2$-torsion points
\begin{align*}
    T_0 &\coloneqq [0\mathbin{:}0\mathbin{:}1], & T_1 &\coloneqq [1\mathbin{:}0\mathbin{:}1], & T_2 &\coloneqq [\lambda\mathbin{:}0\mathbin{:}1],
\end{align*}
and the base point is $O = [0\mathbin{:}1\mathbin{:}0]$. By Theorem~\ref{rmk:AtiyahClass} (b), any vector bundle $\E$ on $C$ is of the form $\bigoplus_{i=1}^n \F_{r_i} \otimes \L_i$ for some degree zero line bundles $\L_i$. Depending on these $\L_i$ we have an open cover $\{U, V\}$ of $C$ which trivializes $\E$, whose construction we review.

Each non-trivial $\L_i$ is isomorphic to $\O_C(P_i - O)$ for a unique $P_i \in C \backslash \{O\}$. Let $L_i$ be a linear form in $\Qbar[X, Y, Z]$ such that the line $L_i = 0$ passes through $P_i$ and avoids all of $O, T_0, T_1, T_2$, and $P_j$ for $i \neq j$, (unless $P_i = T_k$ for some $k$). This is possible because this is a finite set of points to avoid. The line $L_i = 0$ passes through two additional points $Q_i$ and $R_i$ on the curve $C$. We define
\begin{align*}
    U &\coloneqq C \mathbin{\backslash} \left(\{O\} \cup \{Q_i, R_i\} \right), \\
    V &\coloneqq C \mathbin{\backslash} \left(\{T_0, T_1, T_2\} \cup \{P_i\} \right).
\end{align*}
By construction, this is an open cover. We highlight the vital features of this cover.

\begin{lemma}
\label{lemma:trivialization}
Let $\{U,V\}$ be an open cover of $C$ such that $O \not\in U$ and the $2$-torsion points $T_0,T_1,T_2 \not\in V$.
\begin{enumerate}
    \item[(i)] The intersection of $\O_C(U)$ and $\O_C(V)$ as subrings of $\O_C(U \cap V)$ is $\Qbar$.
    \item[(ii)] There is an element $\omega \in \O_C(U \cap V)$ such that $\omega \not\in \O_C(U)$ and $\omega \not\in \O_C(V)$.
    \item[(iii)] There is no element $f \in \O_C(U)$ such that $f + \omega \in \O_C(V)$. In particular, if $f \in \O_C(U)$ and $c \in \Qbar$ satisfy $f + c\omega \in \O_C(V)$, then the number $c$ equals zero.
\end{enumerate}
\end{lemma}

\begin{proof}
Since $C$ is projective, any rational function regular on all of $C$ is constant, proving (i). For (ii), set $\omega \coloneqq \frac{X^2}{Y Z}$. We claim that
\begin{equation*}
    \div \omega = 3T_0 - T_1 - T_2 - O.
\end{equation*}
Indeed, if $X = 0$, then equation (\ref{eq:legendre}) becomes $ZY^2 = 0$ so we get a zero at $O$ and a zero of order two at $T_0$. If $Y = 0$, then (\ref{eq:legendre}) yields $X(X - Z)(X -\lambda Z) = 0$ which has zeroes at $T_0$, $T_1$ and $T_2$. Finally, if $Z = 0$, then (\ref{eq:legendre}) becomes $X^3 = 0$ which gives a zero of order 3 at $O$. Putting these data together recovers the divisor $\div \omega$. By the initial assumption, $\omega$ is not regular on either $U$ or $V$, but is regular on $U \cap V$, so we are done.

For (iii), since $\Pic^0(C) = \{[P-O] \mathbin{:} P \in C\}$, there cannot be rational functions $f \in \Qbar(C)$ such that $\text{div }f = P - O$ for some $P \neq O$. Now suppose $f \in \O_C(U)$. If the valuation (order of vanishing) of $f$ at $O$ is not $-1$, then $f + c\omega$ has valuation at most $-1$ at $O$ and hence isn't regular on $V$. On the other hand, if $f$ has exactly valuation $-1$ at $O$, then by our observation above, it must have at least one additional pole on $C \backslash U \subset V$, and consequently $f + c\omega$ isn't regular on $V$. 
\end{proof}

For each summand $\F_{r_i} \otimes \L_i$ we obtain an explicit choice of transition matrix from $U$ to $V$ of the form
\begingroup
\renewcommand*{\arraystretch}{0.5}
\begin{align}
    M_i := \frac{\omega}{L_i} \begin{bmatrix} 1 & \omega & 0 & \cdots & 0 \\
                    0 & 1 & \omega & \cdots & 0 \\ 
                    \vdots & ~ & \ddots & ~ & \vdots \\
                    ~ & ~ & ~ & 1 & \omega \\
                    0 & ~ & \cdots & 0 & 1
    \end{bmatrix} \in \O_C(U \cap V)^{r_i}
\end{align}
\endgroup
where if $\L_i \cong \O_C$ then $L_i \coloneqq \omega$ and otherwise $L_i$ is the linear form we chose in the construction of $U$ and $V$. The Cartier divisor $\{(U,L_i), (V,\omega)\}$ represents the Weil divisor $P_i-O$ and so $\frac{\omega}{L_i}$ is a transition function from $U$ to $V$ for $\L_i$. The transition matrix $M$ of the bundle $\E = \bigoplus_{i=1}^n \F_{r_i} \otimes \L_i$ is then the block diagonal matrix with each $M_i$ on the diagonal. See \cite[Theorem~3.19]{zotine2017thesis} for a proof that this is a transition matrix for $\E$.

We explicitly write the compatibility conditions in Proposition~\ref{prop:compatibilityconditions}. By Remark~\ref{rem:compatibility}, the transition function for $\Sym^d \E$ with respect to the open cover $\{U, V\}$ is given by the action $\bigl( \Sym^d M \bigr) \bigl( F(t_0, t_1, \ldots, t_r) \bigr) \coloneqq F \bigl( M^\text{\sffamily T} \cdot (t_0, t_1, \ldots, t_r) \bigr)$. For example, if $\E = \F_2 \oplus (\F_3 \otimes \L)$ and $\alpha$ denotes a transition function from $U$ to $V$ for $\L$, then for $F \in \O_C(U)[t_0,t_1,t_2,t_3]$ we have
\begin{align*}
    \bigl(\Sym^d M\bigr)(F) &= F(t_0, \omega t_0 + t_1, \alpha t_2, \alpha \omega t_2 + \alpha t_3, \alpha\omega t_3 + \alpha t_4).
\end{align*}
\begin{lemma}
\label{lemma:explicitcompconds}
Let $\E = \F_{r+1} \otimes \L$ be an indecomposable vector bundle on $C$ of rank $r+1$. A surjective endomorphism of $\PP(\E)$ satisfying
\begin{align}
\begin{tikzcd}[ampersand replacement = \&]
\PP(\E) \arrow[r, "f"] \arrow[d, "\pi" left] \& \PP(\E) \arrow[d, "\pi" right] \\
C \arrow[r, "g"] \& C
\end{tikzcd}
\tag{$\ddagger$}
\end{align}
and having degree $d$ on the fibres of $\pi$ is equivalent to specifying $r+1$ degree $d$ polynomials $F_0, F_1, \ldots, F_r \in \O_C(U)[t_0,t_1,\ldots,t_r]$ which do not vanish at a common point in $\PP^r$, and $r+1$ degree $d$ polynomials $G_0, G_1, \ldots, G_r \in \O_C(V)[t_0,t_1,\ldots,t_r]$ which do not vanish at a common point in $\PP^r$ such that $\beta \bigl(\Sym^d M\bigr)(F_0) = \gamma G_0$ and $\beta \bigl(\Sym^d M\bigr)(F_i) = \gamma (G_i + \omega G_{i-1})$ for $i = 1, \ldots, r$, where $\beta, \gamma$ are transition functions for $\pi^* \B, g^* \L$ respectively, and $\omega$ is from Lemma~\ref{lemma:trivialization} (ii).
\end{lemma}

\begin{proof}
Let $\B$ be as in Proposition~\ref{prop:fibreProp} and refine the open cover $\{U,V\}$ so that it trivializes $\pi^* \B$ and $g^*\E$ as well. Our construction of the open cover allows one to use only two open sets to accomplish this. Proposition~\ref{prop:compatibilityconditions} gives degree $d$ polynomials 
\begin{align*}
    F_0, F_1, \ldots, F_r \in \O_C(U)[t_0,t_1,\ldots,t_r] & & \text{and} & & G_0, G_1, \ldots, G_r \in \O_C(V)[t_0,t_1,\ldots,t_r].
\end{align*}
These polynomials define the maps $\theta_U$ and $\theta_V$.

Therefore, the commutativity of the diagram in Proposition~\ref{prop:compatibilityconditions} is equivalent to the following products of matrices being equal
\begin{equation*}
     \beta \cdot \begin{bmatrix} \bigl( \Sym^d M \bigr) F_0 & \bigl( \Sym^d M \bigr) F_1 & \cdots & \bigl( \Sym^d M \bigr) F_r \end{bmatrix} = \begin{bmatrix} G_0 & G_1 & \cdots & G_r \end{bmatrix} \cdot (\gamma M),
\end{equation*}
where $\beta$ is a transition function for $\pi^* \B$. Since we have an explicit formula for the transition matrix $M$, we get the compatibility conditions above as desired.
\end{proof}

\section{Endomorphisms of Non-Split Bundles}
\label{sec:nonsplit}

In this section, we show that if $\E$ is not a direct sum of line bundles, then $\PP(\E)$ has no surjective endomorphism with degree greater than one on the fibres of $\pi$. Our strategy is to use Lemma~\ref{lemma:explicitcompconds} and show that each of the $F_i$ vanish at a common point. We start by assuming $\E = \F_{r+1}$ for some positive integer $r$. In this case, we have that $\L \cong g^*\L \cong \O_C$, and $\B \cong \O_C$ by Lemma~\ref{lemma:Bistrivial}. Consequently the first condition in Lemma~\ref{lemma:explicitcompconds} implies $\bigl(\Sym^d M\bigr)(F_0) = G_0 \in \O_C(V)$ because the transition functions satisfy $\beta = \gamma = 1$. This condition implies that many of the coefficients of $F$ must vanish; see Proposition~\ref{prop:vanishingcoeffs} and Example~\ref{ex:vanishing} below.

Throughout the section, we need to extract coefficients from polynomials, so we fix the following notation. Given a degree $d$ polynomial $F \in \O_C(U)[t_0, t_1, \ldots, t_r]$, denote the coefficient of the monomial $t^{\textbf{u}}$ in $F$ as $a_{\textbf{u}}$, where $\textbf{u} \in \NN^{r+1}$ is an exponent vector with component sum is $d$. We use the notation $\bigl[ t^{\textbf{u}} \bigr] F \coloneqq a_{\textbf{u}}$, as in \cite[p.~673]{stanley1999combinatorics}, to denote the coefficient extraction operator. The coefficients of $\bigl(\Sym^d M \bigr)(F)$ are some polynomials in the $a_{\textbf{u}}$ and so it also makes sense to consider the composition of coefficient extraction $\bigl[ a_{\textbf{v}} \bigr] \bigl[ t^{\textbf{u}} \bigr] (\Sym^d M)(F)$. For example, in Equation~\eqref{eqn:coeff-extraction-example} below, the coefficient of $a_{(0,0,0,0,1,6)}$ in the coefficient of $t_5^5 t_4 t_3$ is $6\omega^2$, so we would write $\bigl[ a_{(0,0,0,0,1,6)} \bigr] \bigl[ t_5^5 t_4 t_3 \bigr] (\Sym^d M)(F) = 6 \omega^2$.

In general it is useful to characterize when the coefficient $a_{\textbf{v}}$ appears in the coefficient of $t^\textbf{u}$. To characterize this, consider the integral vectors
\begin{equation*}
    \alpha_i := \textbf{e}_{i+1} - \textbf{e}_i = \begin{bmatrix} 0 & \cdots & -1 & 1 & \cdots & 0 \end{bmatrix}^\text{\sffamily T}, \text{ for any } i = 0,1,\ldots,r-1.
\end{equation*}
\begin{lemma}
\label{lemma:whichcoeffs}
Let $F \in \O_C(U)[t_0,t_1,\ldots,t_r]$ and $\emph{\textbf{u}, \textbf{v}} \in \NN^{r+1}$ be exponent vectors with component sum $d$. We have $\bigl[ a_{\emph{\textbf{v}}} \bigr] \bigl[ t^{\emph{\textbf{u}}} \bigr] \bigl(\Sym^d M\bigr)(F) \neq 0$ if and only if there exist $c_i \in \NN$ such that
\begin{align*}
    \emph{\textbf{u}} + \sum_{i=0}^{r-1} c_i \alpha_i &= \emph{\textbf{v}} & \text{and} & & 0 \leq c_i \leq v_i \text{ for } 0 \leq i \leq r-1.
\end{align*}
Moreover, if the above holds, then we have
\begin{equation*}
    \bigl[ a_{\emph{\textbf{v}}} \bigr] \bigl[ t^{\emph{\textbf{u}}} \bigr] \bigl(\Sym^d M\bigr)(F) = \prod_{i=0}^{r-1} \binom{v_i}{c_i} \omega^{c_i}.
\end{equation*}
\end{lemma}

\begin{proof}
Since $\textbf{u}$ and $\textbf{v}$ both have component sum $d$, their difference lies in the hyperplane spanned by the $\alpha_i$. We can explicitly compute $\textbf{v} - \textbf{u} = \sum_{i=1}^r c_i \alpha_i$ where
\begin{align*}
c_i &\coloneqq v_i - u_i + v_{i-1} - u_{i-1} \text{ for }i = 1, 2, \ldots, r-1, & c_r &\coloneqq v_r - u_r.
\end{align*} 
The coefficient $a_{\textbf{v}}$ appears in $\bigl[ t^{\textbf{u}} \bigr] \bigl(\Sym^d M\bigr)(F)$ if and only if $t^{\textbf{u}}$ appears in the expansion of the product of binomials
\begin{equation*}
    \bigl(\Sym^d M\bigr)(a_{\textbf{v}} t^{\textbf{v}}) = a_{\textbf{v}} (t_r + \omega t_{r-1})^{v_r} (t_{r-1} + \omega t_{r-2})^{v_{r-1}} \cdots (t_1 + \omega t_0)^{v_1} t_0^{v_0}.
\end{equation*}
In particular, we need to be able to choose $t_r^{u_r}$ from the first binomial, which can be done in $\binom{v_r}{u_r} = \binom{v_r}{c_r}$ ways. Generally, we need to be able to choose $t_i^{u_i}$ from the $i$th binomial, but excess monomials from the $(i-1)$st binomial are also copies of $t_i$, and so there are $\binom{v_i}{u_i - (v_{i-1} - u_{i-1})} = \binom{v_i}{c_i}$ ways to choose our copies of $t_i$. This tells us that in order for each binomial coefficient to be nonzero, we must have $0 \leq c_i \leq v_i$ for $1 \leq i \leq r$. Moreover, excess copies always come with a factor of $\omega$, implying the additional remark.
\end{proof}

Here is an example demonstrating the structure of the argument in Proposition~\ref{prop:vanishingcoeffs}.

\begin{example}
\label{ex:vanishing}
Fix $r = 5$ and $d = 7$. Let $F$ satisfy $\bigl(\Sym^d M\bigr)(F) \in \O_C(V)[t_0, t_1, \ldots, t_5]$. It has the form (using the reverse lexicographic monomial ordering)
\begin{equation*}
    F = a_{(0,0,0,0,0,7)} t_5^7 + a_{(0,0,0,0,1,6)} t_5^6 t_4 + a_{(0,0,0,1,0,6)} t_5^6 t_3 + \cdots + a_{(7,0,0,0,0,0)} t_0^7,
\end{equation*}
where each coefficient lies in $\O_C(U)$. Using Lemma~\ref{lemma:whichcoeffs}, we can examine monomials to impose conditions on the coefficients. For instance, we see that
\begin{equation*}
    \bigl[ t_5^7 \bigr] \bigl(\Sym^d M\bigr)(F) = a_{(0,0,0,0,0,7)},
\end{equation*}
because adding any $\alpha_j$ to the exponent vector $\begin{bmatrix} 0 & 0 & 0 & 0 & 0 & 7 \end{bmatrix}$ introduces a negative coefficient, which is not allowed. Consequently, we get that $a_{(0,0,0,0,0,7)} \in \O_C(U) \cap \O_C(V)$, and hence is in $\Qbar$ by Lemma~\ref{lemma:trivialization} (i). Examining the coefficient of $t_5^6 t_4$ in $\bigl(\Sym^d M\bigr)(F)$, Lemma~\ref{lemma:whichcoeffs} gives
\begin{equation*}
    \bigl[ t_5^6 t_4 \bigr] \bigl(\Sym^d M\bigr)(F) = a_{(0,0,0,0,1,6)} + 7\omega a_{(0,0,0,0,0,7)},
\end{equation*}
since we either don't add any of the $\alpha_j$, or can add one copy of $\alpha_r$. We had just seen that $a_{(0,0,0,0,0,7)} \in \Qbar$, but since $\omega \not \in \O_C(V)$ (Lemma~\ref{lemma:trivialization} (iii)) we must conclude that $a_{(0,0,0,0,0,7)} = 0$. Consequently we can now apply Lemma~\ref{lemma:trivialization} (i) to get $a_{(0,0,0,0,1,6)} \in \Qbar$. Examining the coefficient of $t_5^5 t_4^2$, we get
\begin{align*}
    \bigl[ t_5^5 t_4^2 \bigr] \bigl(\Sym^d M\bigr)(F) &= a_{(0,0,0,0,2,5)} + 6\omega a_{(0,0,0,0,1,6)} + \binom{7}{2} \omega^2 a_{(0,0,0,0,0,7)} \\
    &= a_{(0,0,0,0,2,5)} + 6\omega a_{(0,0,0,0,1,6)}.
\end{align*}
Applying Lemma~\ref{lemma:trivialization} (iii), we get that $a_{(0,0,0,0,1,6)} = 0$, then (i) implies $a_{(0,0,0,0,2,5)} \in \Qbar$. Continuing inductively, we can conclude that
\begin{equation*}
    a_{(0,0,0,0,0,7)}, a_{(0,0,0,0,1,6)}, a_{(0,0,0,0,2,5)}, a_{(0,0,0,0,3,4)}, a_{(0,0,0,0,4,3)}, a_{(0,0,0,0,5,2)}, ~\text{and}~ a_{(0,0,0,0,6,1)}
\end{equation*}
are all zero. This concludes our ``base sweep''. To continue, we consider the coefficient of $t_5^6 t_3$:
\begin{equation*}
    \bigl[ t_5^6 t_3 \bigr] \bigl(\Sym^d M\bigr)(F) = a_{(0,0,0,1,0,6)} + 6\omega a_{(0,0,0,0,1,6)}.
\end{equation*}
We already know $a_{(0,0,0,0,1,6)} = 0$, so Lemma~\ref{lemma:trivialization} (i) gives $a_{(0,0,0,1,0,6)} \in \Qbar$. Now we perform a similar induction as before. We consider
\begin{align}
\label{eqn:coeff-extraction-example}
    \bigl[ t_5^5 t_4 t_3 \bigr] \bigl(\Sym^d M\bigr)(F) &= a_{(0,0,0,1,1,5)} + \omega \bigl( 6 a_{(0,0,0,1,0,6)} + a_{(0,0,0,0,2,5)} \bigr) + 6 \omega^2 a_{(0,0,0,0,1,6)} \\
    &= a_{(0,0,0,1,1,5)} + 6\omega a_{(0,0,0,1,0,6)} \nonumber.
\end{align}
Lemma~\ref{lemma:trivialization} (iii) implies $a_{(0,0,0,1,0,6)} = 0$, followed by (i) to imply $a_{(0,0,0,1,1,5)} \in \Qbar$.

Overall, we perform double induction to show that coefficients vanish in the order: (0,0,0,0,0,7), then (reading left to right, top to bottom)
\begin{align*}
    \begin{matrix} 
    (0,0,0,0,1,6), (0,0,0,0,2,5), (0,0,0,0,3,4), (0,0,0,0,4,3), (0,0,0,0,5,2), (0,0,0,0,6,1), \\
    (0,0,0,1,0,6), (0,0,0,1,1,5), (0,0,0,1,2,4), (0,0,0,1,3,3), (0,0,0,1,4,2), (0,0,0,1,5,1), \\
    (0,0,1,0,0,6), (0,0,1,0,1,5), (0,0,1,0,2,4), (0,0,1,0,3,3), (0,0,1,0,4,2), (0,0,1,0,5,1), \\
    (0,1,0,0,0,6), (0,1,0,0,1,5), (0,1,0,0,2,4), (0,1,0,0,3,3), (0,1,0,0,4,2), (0,1,0,0,5,1), \\
    (1,0,0,0,0,6), (1,0,0,0,1,5), (1,0,0,0,2,4), (1,0,0,0,3,3), (1,0,0,0,4,2), (1,0,0,0,5,1).
    \end{matrix}
\end{align*}
Ultimately, the vanishing of the first column of this list is what we are interested in.
\end{example}

\begin{proposition}
\label{prop:vanishingcoeffs}
Denote $M$ as the transition matrix of $\F_{r+1}$, let $F \in \O_C(U)[t_0,t_1,\ldots,t_r]$ be a polynomial of degree $d \geq 2$ with $\bigl(\Sym^d M\bigr)(F) = F(Mt) \in \O_C(V)[t_0, \ldots, t_r]$, and suppose $\emph{\textbf{u}} = (u_0,u_1,\ldots,u_r) \in \NN^{r+1}$ is an exponent vector with component sum $d$. Whenever
\begin{equation*}
    \emph{\textbf{u}} = (0, 0, \ldots, 0, k, d-k-1) + \textbf{\emph{e}}_i
\end{equation*}
for some integers $0 \leq k \leq d-2$ and $0 \leq i \leq r$, we have $\bigl[ t^{\emph{\textbf{u}}} \bigr] F = 0$ and $\bigl[ t^{\emph{\textbf{u}}-\alpha_0} \bigr]F \in \Qbar$.
\end{proposition}

\begin{proof}
We perform a double induction on $k$ and $r-i$. For the base case $k = 0$ and $i = r$, we have $\textbf{u} = (0, 0, \ldots, 0, d)$. Examine the coefficient of $t^{\textbf{u}}$ in $\bigl(\Sym^d M\bigr)(F)$. By Lemma~\ref{lemma:whichcoeffs}, the only coefficient which appears is $a_{\textbf{u}}$, so that $a_{\textbf{u}} \in \mathcal{O}_C(V)$ and Lemma~\ref{lemma:trivialization} (i) implies that $a_{\textbf{u}} \in \Qbar$. Examining the coefficient of $t^{\textbf{u} - \alpha_r}$ in $\bigl(\Sym^d M\bigr)(F)$, Lemma~\ref{lemma:whichcoeffs} gives
\begin{equation*}
    \bigl[ t^{\textbf{u} - \alpha_r} \bigr] \bigl(\Sym^d M\bigr)(F) = a_{\textbf{u} - \alpha_r} + d\omega a_{\textbf{u}}.
\end{equation*}
Hence Lemma~\ref{lemma:trivialization} (iii) tells us that $a_{\textbf{u}} = 0$ and (i) tells us that $a_{\textbf{u} - \alpha_0} \in \Qbar$, establishing the base case.

Now fix $0 \leq i \leq r$ and $0 \leq k \leq d-2$. Assume that for any $0 \leq r - i' \leq r - i$ and $0\leq k' \leq k$ satisfying $0 \leq r - i' + k' < r- i + k$ we have that whenever
\begin{equation*}
\textbf{v} = (0, 0, \ldots, 0, k', d-k'-1) + \textbf{e}_{i'},
\end{equation*}
we get $a_{\textbf{v}} = 0$ and $a_{\textbf{v} - \alpha_0} \in \Qbar$. Set $\textbf{u} = (0, 0, \ldots, 0, k, d-k-1) + \textbf{e}_i$. First, we claim $a_{\textbf{u}} \in \Qbar$. If $k = 0$, then examine the coefficient of $t^{\textbf{u}}$ in $\bigl(\Sym^d M\bigr)(F)$. By Lemma~\ref{lemma:whichcoeffs}, we get
\begin{equation*}
    \bigl[ t^{\textbf{u}} \bigr] \bigl(\Sym^d M\bigr)(F) = a_{\textbf{u}} + \omega a_{\textbf{u}+\alpha_i},
\end{equation*}
because $\textbf{u} + \alpha_j$ has nonnegative entries if and only if $j = i-1$. Furthermore, we have $\textbf{u} + \alpha_i = (0, 0, \ldots, 0, d-1) + \textbf{e}_{i+1}$, so we can apply the inductive hypothesis for $i' = i+1$ and $k = 0$ to conclude that $a_{\textbf{u}+\alpha_i} = 0$. Applying Lemma~\ref{lemma:trivialization} (i), we get that $a_{\textbf{u}} \in \Qbar$. Otherwise, $k > 0$ and we can apply the inductive hypothesis for $i' = i$ and $k' = k-1$ to conclude $a_{\textbf{u}} \in \Qbar$. 

Next, we claim that 
\begin{equation*}
[t^{\textbf{u} - \alpha_r}] = a_{\textbf{u} - \alpha_r} + \omega(d-k) a_{\textbf{u}}.
\end{equation*}
If $k = 0$, this follows directly from Lemma~\ref{lemma:trivialization}. Otherwise, assume $k > 1$. Examine the coefficient of $t^{\textbf{u} - \alpha_r}$ in $\bigl(\Sym^d M\bigr)(F)$. For the sake of clarity, we apply Lemma~\ref{lemma:whichcoeffs} more explicitly in this case. We take our input exponent $\textbf{u} - \alpha_r$, and add multiples of the vectors $\alpha_\ell$ in such a way that the resulting vector still has nonnegative entries. The first term we get is $\textbf{u} - \alpha_r$ trivially, by not adding any of the $\alpha_\ell$. In other words, each $c_j$ in Lemma~\ref{lemma:whichcoeffs} is zero. It follows that $a_{\textbf{u} - \alpha_r}$ appears, and has a coefficient of $1$. Next we try to add just one $\alpha_\ell$, but there are only two possibilities corresponding to the entries of $\textbf{u}$ with a nonzero entry on the left, which are $\alpha_r$ and $\alpha_i$. This gives $a_{\textbf{u}}$ appears with coefficient $\binom{d-k}{1} \omega$ and $a_{\textbf{u} - \alpha_r + \alpha_i}$ appears with coefficient $\binom{1}{1} \omega$.

Generally, when we try to add $j$ of the $\alpha_\ell$ to $\textbf{u}$ for $2 \leq j \leq k+1$, there are only two options: either we add $j\alpha_r$ or $(j-1)\alpha_r + \alpha_i$ (in the case $j=k+1$ only the latter is an option). Hence we obtain the terms
\begin{align*}
&\binom{d-k-1+j}{j} \omega^j a_{\textbf{u} + (j-1)\alpha_r}& &\text{and}& &\binom{d-k-2+j}{j-1} \omega^j a_{\textbf{u} + (j-2)\alpha_r + \alpha_i}.
\end{align*}
Summarizing, we obtain
\begin{align*}
    \bigl[ t^{\textbf{u} - \alpha_r} \bigr] \bigl(\Sym^d M\bigr)(F) &= a_{\textbf{u} - \alpha_r} + \omega ((d-k) a_{\textbf{u}} + a_{\textbf{u} - \alpha_r + \alpha_i}) + \binom{d-2}{k} \omega^{k+1} a_{\textbf{u} + (k-1)\alpha_r + \alpha_i} \\
    &+ \sum_{j=2}^{k} \omega^j \left( \binom{d-k-1+j}{j} a_{\textbf{u} + (j-1)\alpha_r} + \binom{d-k-2+j}{j-1} a_{\textbf{u} + (j-2)\alpha_r + \alpha_i} \right).
\end{align*}
We apply the inductive hypothesis for $i' = i, k' = k-j$, and, if $i > 0$, for $i' = i-1, k' = k-j$ to obtain
\begin{equation*}
    a_{\textbf{u} - \alpha_r + \alpha_i} = a_{\textbf{u} + (j-1)\alpha_r} = a_{\textbf{u} + (j-2)\alpha_r + \alpha_i} = 0
\end{equation*}
for each $2 \leq j \leq k+1$. Therefore, we are left with
\begin{equation*}
    \bigl[ t^{\textbf{u} - \alpha_r} \bigr] \bigl(\Sym^d M\bigr)(F) = a_{\textbf{u} - \alpha_r} + \omega (d-k) a_{\textbf{u}}
\end{equation*}
as claimed. Finally, we apply Lemma~\ref{lemma:trivialization} (iii) to get $a_{\textbf{u}} = 0$ and (i) to get $a_{\textbf{u} - \alpha_0} \in \Qbar$, completing the induction.
\end{proof}

With these vanishing conditions on the coefficients, we can guarantee that the polynomials satisfying the compatibility conditions of Lemma~\ref{lemma:explicitcompconds} must have a common zero.

\begin{proposition}
\label{prop:commonzero}
Suppose
\begin{align*} 
F_0, F_1, \ldots, F_r \in \O_C(U)[t_0,t_1,\ldots,t_r] & & \text{and} & & G_0, G_1, \ldots, G_r \in \O_C(V)[t_0,t_1,\ldots,t_r]
\end{align*}
satisfy the compatibility conditions
\begin{align*}
    \bigl(\Sym^d M\bigr)(F_0) = G_0 & & \text{and} & & \bigl(\Sym^d M\bigr)(F_i) = G_i + \omega G_{i-1}
\end{align*}
for $i = 1, 2, \ldots, r$. The polynomials $F_0, F_1, \ldots, F_r$ have a common zero at $[0:0:\cdots:0:1]$.
\end{proposition}

\begin{proof}
Using strong induction on $i$, we show that $\bigl[ t_r^{d-1} t_j \bigr] F_i = 0$ for every $0 \leq i \leq r$ and $i \leq j \leq r$, and $\bigl[ t_r^{d-1} t_j \bigr] G_i = 0$ for each $i+1 \leq j \leq r$. For the base case $i = 0$, we have $\bigl(\Sym^d M\bigr)(F_0) = G_0 \in \O_C(V)[t_0,t_1,\ldots,t_r]$. Hence Proposition~\ref{prop:vanishingcoeffs} implies that $\bigl[ t_r^{d-1} t_j \bigr] F_0 = 0$ for each $0 \leq j \leq r$. On the other hand, Lemma~\ref{lemma:whichcoeffs} gives us that for each $1 \leq j \leq r$, 
\begin{equation*}
\bigl[ t_r^{d-1} t_j \bigr] G_0 = \bigl[ t_r^{d-1} t_j \bigr] \bigl(\Sym^d M\bigr)(F_0) = \bigl[ t_r^{d-1} t_j \bigr] F_0 + \bigl[ t_r^{d-1} t_{j-1} \bigr] F_0.
\end{equation*}
We just proved that $\bigl[ t_r^{d-1} t_j \bigr] F_0 = 0$ for each $0 \leq j \leq r$, so we conclude $\bigl[ t_r^{d-1} t_j \bigr] G_0 = 0$ for each $1 \leq j \leq r$.

Now fix $1 \leq i \leq r-1$. Assume $\bigl[ t_r^{d-1} t_j \bigr] F_i = 0$ for each $i \leq j \leq r$, and $\bigl[ t_r^{d-1} t_j \bigr] G_i = 0$ for each $i+1 \leq j \leq r$. Since $\bigl(\Sym^d M\bigr)(F_{i+1}) = G_{i+1} + \omega G_i$, we obtain for each $i+1 \leq j \leq r$
\begin{equation*}
    \bigl[ t_r^{d-1} t_j \bigr] \bigl(\Sym^d M\bigr)(F_{i+1}) = \bigl[ t_r^{d-1} t_j \bigr] G_{i+1}.
\end{equation*}
Set $F = \sum_{j=i+1}^r t_r^{d-1} t_j \bigl[ t_r^{d-1} t_j \bigr] F_{i+1}$. We get $\bigl(\Sym^d M\bigr)(F) \in \O_C(V)[t_0,t_1,\ldots,t_r]$ and hence we can apply Proposition~\ref{prop:vanishingcoeffs} to see that $\bigl[ t_r^{d-1} t_j \bigr] F_{i+1} = 0$ for each $i+1 \leq j \leq r$. Similarly to the base case, Lemma~\ref{lemma:whichcoeffs} gives us that for each $i+2 \leq j \leq r$
\begin{equation*}
\bigl[ t_r^{d-1} t_j \bigr] G_{i+1} = \bigl[ t_r^{d-1} t_j \bigr] \bigl(\Sym^d M\bigr)(F) = \bigl[ t_r^{d-1} t_j \bigr] F + \bigl[ t_r^{d-1} t_{j-1} \bigr] F,
\end{equation*}
which is then also zero. This establishes the induction, which shows that $\bigl[ t_r^d \bigr] F_i = 0$ for each $0 \leq i \leq r$. Therefore all of the $F_i$ have a common zero at $[0:0:\dotsb:0:1]$.
\end{proof}

Using Proposition~\ref{prop:commonzero}, we show that the assumptions of Lemma~\ref{lemma:explicitcompconds} are impossible for $d > 1$, proving Theorem~\ref{intro:main1}.

\begin{theorem}[\ref{intro:main1}]
\label{thm:nonsplitnoendos}
Suppose $\E = \bigoplus_{i=1}^n \F_{r_i + 1} \otimes \L_i$, with $\L_i$ being degree zero line bundles and $r_i \geq 0$ for each $i$. If there is some $j$ such that $r_j \geq 1$, then the bundle $\PP(\E)$ has no surjective endomorphisms of degree greater than one on the fibres of $\pi$.
\end{theorem}

\begin{proof}
Without loss of generality, assume $r_1 \geq 1$, $r_1 \geq r_i \geq 0$ for every $2 \leq i \leq n$, and $\L_1 = \O_C$. Suppose we are able to construct an endomorphism with degree greater than one on the fibres. We restrict our attention to the $i$th summand $\F_{r_i+1} \otimes \L_i$ of $\E$. Lemma~\ref{lemma:explicitcompconds} gives degree $d$ polynomials 
\begin{align*}
    F_0,\dotsc,F_{r_i} \in \O_C(U)[t_0,t_1,\dotsc,t_r] & & \text{and} & & G_0,\dotsc, G_{r_i} \in \O_C(V)[t_0,t_1,\dotsc,t_r] 
\end{align*}
where $r$ is the total rank of $\E$. By substituting in $t_j = 0$ for $r_1 < j \leq r$, we obtain polynomials in the variables $t_0,t_1,\ldots,t_{r_1}$ instead which we denote $\widetilde{F}_j$ and $\widetilde{G}_j$. The compatibility conditions from Lemma~\ref{lemma:explicitcompconds} become
\begin{equation*}
    \beta \bigl(\Sym^d M\bigr)(\widetilde{F}_j) = \begin{cases} \gamma_i \widetilde{G}_j & j = 0 \\ \gamma_i (\widetilde{G}_j + \omega \widetilde{G}_{j-1}) & 1 \leq j \leq r_i \text{ (if $r_i \neq 0$)} \end{cases}
\end{equation*}
where $\beta$ and $\gamma_i$ are the transition functions for $\pi^*\B$ and $g^* \L_i$ respectively. If $r_i = 0$, then  In particular, we obtain
\begin{equation*}
    \bigl[ t_0^d \bigr] \bigl(\Sym^d M\bigr)(\widetilde{F}_j) = \bigl[ t_0^d \bigr] \widetilde{F}_j = \begin{cases} \beta^{-1} \gamma_i \bigl[ t_0^d \bigr] \widetilde{G}_j & j = 0 \\ \beta^{-1} \gamma_i (\bigl[ t_0^d \bigr] \widetilde{G}_j + \omega \bigl[ t_0^d \bigr] \widetilde{G}_{j-1}) & 1 \leq j \leq r_i \text{ (if $r_i \neq 0$)} \end{cases}.
\end{equation*}
If $\bigl[ t_0^d \bigr]\widetilde{F}_i=0$ for each $i$ then we are done. Toward a contradiction, suppose we have $\bigl[ t_0^d \bigr] \widetilde{F}_j \neq 0$ for some minimal index $j$. This gives $\bigl[ t_0^d \bigr] \widetilde{F}_j = \beta^{-1} \gamma_i \bigl[ t_0^d \bigr] \widetilde{G}_j$ since $\bigl[ t_0^d \bigr] \widetilde{G}_{j-1} = 0$ by minimality of $j$. Hence we obtain a nonzero section of $\pi^* \B^{-1} \otimes g^* \L_i$, which means that this tensor product isomorphic to the trivial bundle. It follows that we can retroactively pick the transition function $\beta = \gamma_i$. Consequently, all of the compatibility conditions from Lemma~\ref{lemma:explicitcompconds} become equivalent to the compatibility conditions
\begin{align*}
    \bigl(\Sym^d M\bigr)(\widetilde{F}_0) = \widetilde{G}_0 & & \text{and} & & \bigl(\Sym^d M\bigr)(\widetilde{F}_j) = \widetilde{G}_j + \omega \widetilde{G}_{j-1}
\end{align*}
for $1 \leq j \leq r_i$ (as long as $r_i \neq 0$). Now there are two cases. If $r_i = 0$, we can apply Proposition~\ref{prop:vanishingcoeffs} to see that $\bigl[t_r^{d} \bigr] \widetilde{F}_0 = 0$, so that $\widetilde{F}_0$ has a zero at $[0:0:\cdots:0:1]$. On the other hand, if $r_i > 0$, then we can apply Proposition~\ref{prop:commonzero} to see that all of the $\widetilde{F}_j$ must have a common zero at $[0:0:\cdots:0:1] \in \PP^{r_1}$. In either case, the original $F_j$ must have a common zero at $[0:0:\cdots:0:1] \in \PP^r$. Since we only argue in terms of the first $r_1 + 1$ variables, we can apply this argument for every summand. This shows that every polynomial from Lemma~\ref{lemma:explicitcompconds} has a zero at $[0:0:\cdots:0:1]$, contradicting the existence of the surjective endomorphism with degree greater than one on the fibres of $\pi$. 
\end{proof}

\begin{remark}
This argument could be phrased using the representation theory of the unipotent group $\GG_a \subset \operatorname{SL}_2$. The transition matrix of $\F_r$ is of the same structure as the representation $\sym^r V_2$, where $V_2$ is the defining representation of $\GG_a$. By fixing a basis for $\PP^r$, the transition matrix defines a $\GG_a$-action on $\PP^r$, which yields a representation. Proposition~\ref{prop:compatibilityconditions} gives us that a surjective endomorphism of $\F_r$ of degree $d$ defines a $\GG_a$-equivariant morphism $\varphi \colon \sym^r V_2 \rightarrow \sym^d \sym^r V_2$. The Lie algebra of $\GG_a$ is generated by a single element $\partial_z$, and the image of $\sym^r V_2$ under $\varphi$ has to lie in the subspace of $\sym^d \sym^r V_2$ annihilated by $\partial_z^{r+1}$. All of the elements of this subspace must vanish at $[0:0:\cdots:0:1]$. 
\end{remark}

\section{Direct sums of line bundles}
\label{sec:splitcase}

In this section, let $X$ be any smooth projective variety. When $X$ is not an abelian variety, we make use of the Albanese morphism. The Albanese variety of a normal projective variety $X$ is an abelian variety $\operatorname{Alb}(X)$ together with a morphism $\alpha_X\colon X\rightarrow \operatorname{Alb}(X)$ that is initial among all morphisms from $X$ to an abelian variety. The desire to construct $\operatorname{Alb}(X)$ goes back to Weil and was viewed as a task of utmost importance. See \cite{kleiman2014picard} for a historical exposition and details of the construction. The Picard variety $\Pic^0(X)$, which is an abelian variety, is first constructed, and the Albanese variety $\operatorname{Alb}(X)$ is the dual abelian variety of $\Pic^0(X)$. 

We have two correspondences: An algebraically trivial line bundle on $X$ corresponds to a point on $\Pic^0(X)$, and additionally points on $\Pic^0(X)$ correspond to algebraically trivial line bundles on the dual abelian variety $\operatorname{Alb}(X)$. Therefore algebraically trivial line bundles on $\operatorname{Alb}(X)$ correspond to algebraically trivial line bundles on $X$ via pullback by $\alpha_X$; see \cite[Remark~5.25]{kleiman2014picard}.

Let $g\colon X\rightarrow X$ be a surjective morphism of a smooth projective variety. Recall that $\lambda_1(g)$ is the spectral radius of $g^*$ acting on $N^1(X)_\QQ=(\Pic(X)/\Pic^0(X))\otimes_\ZZ \QQ$. It is useful to relate $\lambda_1(g)$ to a notion of spectral radius for $g^* \colon \Pic^0(X) \rightarrow \Pic^0(X)$. For any line bundle $\L \in \Pic^0(X)$, let $V_\L \subset \Pic^0(X) \otimes \QQ$ be the sub-vector space spanned by $(g^*)^n \L$ for all $n \in \NN$.
\begin{proposition}\label{prop:TIRSize}
Let $g \colon X \rightarrow X$ be a surjective endomorphism with $X$ a normal projective variety over $\Qbar$ with surjective Albanese map $\alpha_X$. Let $g^* \colon \Pic^0(X) \rightarrow \Pic^0(X)$ denote the pullback map and $g^*_\QQ$ its extension to $\Pic^0(X) \otimes_\ZZ \QQ$. We have
\begin{enumerate}
    \item[(a)] For any line bundle $\L \in \Pic^0(X)$, the vector space $V_\L$ is finite dimensional.
    \item[(b)] If $\rho(g^*_\QQ, V_\L)$ denotes the spectral radius of $g^*_\QQ$ restricted to $V_\L$, then $\rho(g^*_\QQ, V_\L) \leq \sqrt{\lambda_1(g)}$.
\end{enumerate}
\end{proposition}
\begin{proof}
Let $\alpha_X \colon X\rightarrow \textnormal{Alb}(X)$ be the projection. The surjection $\alpha_X$ induces an isomorphism
\begin{equation*}
    \alpha_X^* \colon \Pic^0 \bigl( \operatorname{Alb}(X) \bigr) \rightarrow \Pic^0(X)
\end{equation*}
on the level of $\ZZ$-modules. 
We also have a commuting diagram
\begin{equation*}
\begin{tikzcd}
    X \arrow[r, "g"] \arrow[d, "\pi" left] & X \arrow[d, "\pi"] \\ \textnormal{Alb}(X) \arrow[r,"h"] & \textnormal{Alb}(X)
\end{tikzcd}
\end{equation*}
by the universal property of the Albanese variety so that $\lambda_1(g)\geq \lambda_1(h)$. Thus we may assume that $X$ is smooth and, in fact, an abelian variety with $g\colon X\rightarrow X$ a surjective morphism. 

Let $\psi_\ZZ \colon H^1(X, \ZZ) \rightarrow H^1(X, \ZZ)$ be the pullback by $g$ acting on $H^1(X,\ZZ)$ and $P$ the characteristic polynomial of $\psi_\ZZ$. Since $P$ is a monic polynomial with integer coefficients, it makes sense to evaluate $P$ on any endomorphism of an abelian group. It follows that $P(\psi_\ZZ) = 0$, and we claim that $P(g^*) = 0$ as well.

Set $W = H^1(X, \O_X)$. From Hodge theory we get $H^1(X, \ZZ) \otimes_\ZZ \CC = W \oplus \overline{W}$ and the action of $\psi_\CC \coloneqq \psi_\ZZ \otimes_\ZZ \operatorname{Id}_\CC$ preserves $W$ and $\overline{W}$. Let $Q$ be the characteristic polynomial of $\psi_\CC$ restricted to $W$. The polynomial $Q$ is of degree $\dim(W)$ with complex coefficients. It follows that the characteristic polynomial of $\psi_\CC$ restricted to $\overline{W}$ is the conjugate $\overline{Q}$ and we have $P = Q \overline{Q}$.

The vector space $W$ is the universal cover of $\Pic^0(X)$, meaning we have a quotient map $\pi \colon W \rightarrow \Pic^0(X)$, and we have a commuting diagram of abelian groups
\begin{equation*}
\begin{tikzcd}
W \arrow[r, "\psi_\CC|_W" above] \arrow[d, "\pi" left] & W \arrow[d, "\pi" right] \\
\Pic^0(X) \arrow[r, "g^*" above] & \Pic^0(X).
\end{tikzcd}
\end{equation*}
Since $P(\psi_\ZZ) = 0$, we also have that $P(\psi_\CC|_W) = 0$ and hence by the commutativity of the diagram we get $P(g^*) = 0$. Therefore for any fixed $\L \in \Pic^0(X)$, the vector space $V_\L$ is in fact finitely generated by $(g^*)^n \L$ for $n = 1, 2, \ldots, \deg(P)$, proving (a).

To see (b), let $M$ denote the maximum modulus of the roots of $P$, and hence $Q$ and $\overline{Q}$ as well. Given any $\L \in \Pic^0(X)$, since $P(g^*) = 0$, we also have $P(g^*|_{V_\L}) = 0$ so that $\rho(g^*_\QQ, V_\L) \leq M$. Now because $X$ is an abelian variety, we have $H^{1,1}(X, \CC) = W \otimes \overline{W}$ and therefore the spectral radius of the pullback by $g^*$ acting on $H^{1,1}(X, \CC)$ is $M^2$. By \cite[Remark 5.8]{sano2017heights} we have that $\lambda_1(g)$ is equal to this spectral radius, so we get $\rho(g^*_\QQ, V_\L) \leq \sqrt{\lambda_1(g)}$ as desired. 
\end{proof}

Suppose $\L_0, \L_1, \ldots, \L_r$ are numerically trivial line bundles on $X$ with $\L_0=\O_X$ and set $\E = \bigoplus_{i=0}^r \L_i$. Suppose we have a morphism $f$ which yields a commutative diagram \eqref{eq:dagger}
\begin{equation*}
\begin{tikzcd}
\PP(\E) \arrow[r, "f"] \arrow[d, "\pi" left] & \PP(\E) \arrow[d, "\pi" right] \\
X \arrow[r, "g"] & X
\end{tikzcd}
\end{equation*}
and assume $f^*\O_{\PP(\E)}(1)\cong \O_{\PP(\E)}(d)\otimes \pi^* \B$ where $d \in \ZZ$ and $\B$ is a numerically trivial line bundle on $X$. Let $\{U_j\}$ be an open cover of $X$ trivializing $\E, g^* \E,$ and $\pi^* \B$, with $\alpha_{i,j,k}, \beta_{j,k}$, and $\gamma_{i,j,k}$ denoting the transition functions of $\L_i, \pi^* \B$, and $g^* \L_i$ from $U_j$ to $U_k$ respectively. The conditions of Proposition~\ref{prop:compatibilityconditions} gives the polynomials $F_{0,j}, \ldots, F_{r,j} \in \O_X(U_j)[t_0,t_1,\ldots,t_r]$ such that for any $i$,
\begin{equation*}
    \beta_{j,k} \bigl(\Sym^d M\bigr)(F_{i,j}) = \beta_{j,k} F_{i,j}(\alpha_{0,j,k} t_0, \ldots, \alpha_{r,j,k} t_r) = \gamma_{i,j,k} F_{i,k}(t_0, \ldots, t_r)
\end{equation*}
for all pairs $j,k$. Fixing a degree vector $\textbf{u} \in \NN^{r+1}$, we can compare coefficients to obtain
\begin{equation*}
    \beta_{j,k} \left( \prod_{\ell=0}^r \alpha_{\ell,j,k} \right) \bigl[ t^{\textbf{u}} \bigr] F_{i,j} = \gamma_{i,j,k} \bigl[ t^{\textbf{u}} \bigr] F_{i,k}.
\end{equation*}
In other words, this gives a global section of the line bundle
\begin{equation*}
g^* \L_i \otimes \pi^* \B^{-1} \otimes \left(\bigotimes_{\ell=0}^r \L_\ell^{\otimes-d_i} \right).
\end{equation*}
In particular, this global section is nonzero if and only if $\bigl[ t^{\textbf{u}} \bigr] F_{i,j} \neq 0$ for some $j$. Consequently, in order for $\PP(\E)$ to have complicated endomorphisms, these line bundles are forced to have interesting collections of global sections. However, if all of $\pi^* \B, \L_0,\ldots, \L_r$ are numerically trivial, then they have no nonzero global sections. It follows that if we ever had $\bigl[ t^{\textbf{u}} \bigr] F_{i,j} \neq 0$ for some $j$, it would force the relation
\begin{align}
\label{eq:keyRel}
g^* \L_i \otimes \pi^* \B^{-1} \otimes \left(\bigotimes_{\ell=0}^r \L_\ell^{\otimes-d_i} \right)=\O_X.
\end{align}
This relation heavily restricts the degree on the fibres of an endomorphism.

\begin{proposition}
\label{prop:spectralradius}
Let $X$ be a normal projective variety over a number field $\KK$ with surjective Albanese map. Let $\L_0, \L_1, \ldots, \L_r$ be numerically trivial line bundles on $X$ with $\L_0=\O_X$ and assume $\L_1$ is non-torsion. Set $\E = \bigoplus_{i=0}^r \L_i$ and let $V$ be the vector space spanned by $\pi^* \B, \L_0, \L_1, \ldots, \L_r$ in $\Pic^0(X)_\QQ$. Whenever we have a commutative square \eqref{eq:dagger} and $f^*\O_{\PP(\E)}(1) \equiv_{\textnormal{lin}} \O_{\PP(\E)}(d) \otimes \pi^* \B$ for some numerically trivial line bundle $\B$, we have that $d = \rho(g^*_\QQ, V)$, where $\rho(g^*_\QQ, V)$ is the spectral radius of $g^*$ restricted to $V$.
\end{proposition}

\begin{proof}
Our first goal is to show that $\pi^* \B$ is torsion. Let $\{U_i\}$ be any open cover of $X$ trivializing $\E$. By Proposition~\ref{prop:compatibilityconditions}, we obtain polynomials over the first open set
\begin{equation*}
    F_0, F_1, \ldots, F_r \in \O_X(U_0)[t_0, t_1, \ldots, t_r]_d,
\end{equation*} 
which do not share a common zero in $\PP^r$. Hence, the monomial $t_0^d$ appears with nonzero coefficient in $F_{p_0}$ for some index $0 \leq p_0 \leq r$, because otherwise the polynomials $F_0, F_1, \ldots, F_r$ would have a common zero. It follows by equation \eqref{eq:keyRel} that
\begin{equation*}
g^* \L_{p_0} \otimes \pi^* \B^{-1} \otimes \L_0^{-d} = g^* \L_{p_0} \otimes \pi^* \B^{-1} = \O_X.
\end{equation*}
In other words, we have $\pi^* \B = g^* \L_{p_0}$. If $\L_{p_0}$ is trivial, then so is $\pi^* \B$, so assume $\L_{p_0}$ is non-trivial. We have that $t_{p_0}^d$ appears in $F_{p_1}$ for some $0 \leq p_1 \leq r$ because otherwise the $F_j$ would have a common zero. Therefore by equation \eqref{eq:keyRel} we have 
\begin{equation*}
    g^* \L_{p_1} \otimes \pi^* \B^{-1} \otimes \L_{p_0}^{-d} = g^* \L_{p_1} \otimes g^* \L_{p_0}^{-1} \otimes \L_{p_0}^{-d} = \O_X
\end{equation*}
so that 
\begin{align}
\label{eq:lbbase}
    g^* \L_{p_1} = g^* \L_{p_0} \otimes \L_{p_0}^d.
\end{align}
In general, for $0 < j \leq r+1$, set $p_j$ to be any index for which $\bigl[ t_{p_{j-1}}^d \bigr] F_{p_j} \neq 0$, so that we have the relation $g^* \L_{p_j} \otimes g^* \L_{p_0}^{-1} \otimes \L_{p_{j-1}}^{-d} = \O_X$ and hence 
\begin{align}
\label{eq:lbgen}
    g^* \L_{p_j} = g^* \L_{p_0} \otimes \L_{p_{j-1}}^d.
\end{align}
Using induction on $j$, we see that
\begin{align}
\label{eq:generalLpj}
    (g^*)^j \L_{p_j} = \bigotimes_{k=0}^j (g^*)^{j-k} \L_{p_0}^{d^k}.
\end{align}
Let $G$ be the subgroup of $\Pic^0(X)$ generated by $\pi^* \B = g^* \L_{p_0}, \L_0, \L_1, \ldots, \L_r$ so that $V = G \otimes \QQ$. It follows by equations \eqref{eq:lbbase} and \eqref{eq:lbgen} that $g^*|_G$ is an endomorphism of $G$. By the pigeonhole--principle, since there are $r+2$ indices $p_j$ and $r+1$ indices $j$, we must have that $p_j = p_{j+\ell}$ for some $0 \leq j \leq r$ and $\ell > 0$. Hence equation \eqref{eq:generalLpj} yields
\begin{equation*}
    (g^*)^\ell (g^*)^j \L_{p_j} = \bigotimes_{k=0}^j (g^*)^{j+\ell-k} \L_{p_0}^{d^k} = \bigotimes_{k=0}^{j+\ell} (g^*)^{j+\ell-k} \L_{p_0}^{d^k} = (g^*)^{j+\ell} \L_{p_{j+\ell}}.
\end{equation*}
Therefore after canceling equal terms, we see that $\L_{p_0}$ is in the kernel of $\sum_{k=j+1}^{j+\ell} d^k (g^*)^{j+\ell-k}$ viewed as a morphism on $G$. If this morphism is an isogeny on all of $\Pic^0(X)$, then since the kernel of every isogeny is finite, we obtain that $\L_{p_0}$ is torsion, and hence so is $\pi^* \B$. On the other hand, if $\sum_{k=j+1}^{j+\ell} d^k (g^*)^{j+\ell-k}$ is not an isogeny and $\L_{p_0}$ is non-trivial, then after rewriting, $g^*$ satisfies
\begin{equation*}
    \sum_{k=0}^{\ell-1} d^{j+\ell-k} (g^*)^k = 0
\end{equation*}
on $G$. In other words, in the vector space $V = G \otimes \QQ$, the characteristic polynomial of $g^*_\QQ$ restricted to $V$ is divisible by the polynomial $\sum_{k=0}^{\ell-1} d^{j+\ell-k} x^k = 0$, whose roots are of magnitude $d$. To see this, observe that
\begin{equation*}
    \sum_{k=0}^{\ell-1} d^{j+\ell-k} x^k = d^{j+1} \sum_{k=0}^{\ell-1} d^{\ell-1-k} x^k = d^{j+1} \frac{x^\ell - d^\ell}{x-d}.
\end{equation*}
Therefore in this case we get $d = \rho(g^*_\QQ, V)$ which proves the proposition. Hence we may assume $\pi^* \B = g^* \L_{p_0}$ is torsion. Now we perform a similar argument by starting with the monomial $t_1^d$ instead. We must have some $0 \leq q_0 \leq r$ such that $\bigl[ t_1^d \bigr] F_{q_0} \neq 0$, so that
\begin{equation*}
g^* \L_{q_0} \otimes \pi^* \B^{-1} \otimes \L_1^{-d} = \O_X.
\end{equation*}
More generally for each $0 < j \leq r+1$ we set $q_j$ to be any index such that $\bigl[ t_{q_j}^d \bigr] F_{q_{j-1}} \neq 0$ and consequently
\begin{equation*}
g^* \L_{q_j} \otimes \pi^* \B^{-1} \otimes \L_{q_{j-1}}^{-d} = \O_X.
\end{equation*}
Restricting to the subgroup $G$, we can tensor by $\QQ$ to kill the torsion. In particular, we are assuming $\pi^* \B$ is torsion, and so we obtain
\begin{align*}
    & g^* \L_{q_0} = \L_1^d & & \text{and} & & g^* \L_{q_j} = \L_{q_{j-1}}^d.
\end{align*}
Induction on $j$ yields $(g^*)^j \L_{q_j} = \L_1^{d^j}$. Again by the pigeonhole--principle, we must have $q_j = q_{j+\ell}$ for some $j,\ell$, and we obtain
\begin{equation*}
    (g^*)^\ell (g^*)^j \L_{q_j} = (g^*)^\ell \L_1^{d^j} = \L_1^{d^{j+\ell}} = (g^*)^{j+\ell} \L_{q_{j+\ell}}.
\end{equation*}
This implies that the line bundle $\L_1$ lies in the kernel of the morphism $d^j (g^*)^{\ell} - d^{j+\ell}$. This morphism cannot be an isogeny because that would imply $\L_1$ is torsion, and hence the characteristic polynomial of $g^*_\QQ$ restricted to $V$ is divisible by $d^j (x^\ell - d^\ell)$. It follows that $d = \rho(g^*_\QQ, V)$ as desired. 
\end{proof}

These tools allow one to prove Theorem~\ref{intro:main2} in the case where at least one of the line bundles is non-torsion.

\begin{theorem}[\ref{intro:main2}]
\label{thm:mainThmAnontorsion}
Let $X$ be a smooth projective variety over $\Qbar$ such that its Mori cone is generated by finitely many numerical classes of curves. Fix $\L_0, \L_1,\ldots, \L_r$ to be numerically trivial line bundles on $X$ with $\L_1$ non-torsion and set $\E = \bigoplus_{i=0}^r \L_i$. Suppose that there is a diagram
\begin{equation*}\xymatrix{\PP(\E)\ar[r]^{f}\ar[d]_\pi &\PP(\E)\ar[d]^\pi\\ X\ar[r]_g & X}
\end{equation*}
with $f$ and $g$ surjective. Then the degree of $f$ on the fibres of $\pi$ is at most $\lambda_1(g)$ and $\lambda_1(f)=\lambda_1(g)$.
\end{theorem}

The first assumption on $X$ is guaranteed if $X$ is a Fano variety.

\begin{proof}
After twisting by $\L_0^{-1}$, we may assume that $\L_0 = \O_X$. By \cite[Corollary 3.4]{misra2022nefcones}, the nef cone of $\PP(\E)$ is generated by $\O_{\PP(\E)}(1)$ and $\pi^* \operatorname{Nef}(X)$. After iterating $f$, we may assume that we have a diagram 
\begin{align}
\begin{tikzcd}[ampersand replacement = \&]
\PP(\E) \arrow[r, "f"] \arrow[d, "\pi" left] \& \PP(\E) \arrow[d, "\pi" right] \\
X \arrow[r, "g"] \& X
\end{tikzcd}
\tag{$\dagger$}
\end{align}
and that $f^*\O_{\PP(\E)}(1) \equiv_{\text{num}} \O_{\PP(\E)}(d)$. This means that we have a numerically trivial line bundle $\B$ with $f^*\O_{\PP(\E)}(1) \equiv_{\text{lin}} \O_{\PP(\E)}(d)\otimes \pi^* \B$ for some $d \geq 1$. Hence by Proposition~\ref{prop:spectralradius} and Proposition~\ref{prop:TIRSize} we have that
\[\lambda_1(f|_\pi) = d = \rho(g^*_\QQ, V) \leq \sqrt{\lambda_1(g)}.\]
Therefore we have
\[\lambda_1(f)=\max\{\lambda_1(g),\lambda_1(f\mid_\pi)\}=\lambda_1(g)\]
as desired.
\end{proof}

\subsection{Kawaguchi--Silverman Conjecture}

Theorem~\ref{thm:mainThmAnontorsion}, together with Corollary~\ref{cor:standard1}, resolves the \hyperref[conj:KSC]{KS}~Conjecture when at least one of the line bundles is non-torsion. When $\E$ is a direct sum of torsion line bundles, we require a different approach. Before proceeding, we sketch the idea. For simplicity, assume that the base variety $X$ is abelian, and $\L_0,\ldots \L_r$ are torsion line bundles on $X$. Let $N$ be chosen so that $L_i^{\otimes N}\cong \O_X$. Let $\E=\bigoplus_{i=0}^r\L_i$ and denote $[N]\colon X\rightarrow X$ as multiplication by $N$. We have a surjective pullback morphism $X\times \PP^r\cong \PP([N]^*\E)\rightarrow \PP(\E)$ where the first isomorphism is due to our choice of $N$. We now lift a surjective morphism $f\colon \PP(\E)\rightarrow \PP(\E)$ to a morphism $h\colon \PP([N]^*\E)\rightarrow \PP([N]^* \E)$  and obtain a diagram
\begin{equation*}
\begin{tikzcd}
    \PP([N]^*\E) \arrow[r, "h"] \arrow[d] & \PP([N]^*\E)\arrow[d] \\ \PP(\E)\arrow[r, "f"] & \PP(\E).
\end{tikzcd}
\end{equation*}
The Kawaguchi--Silverman conjecture is known for $ \PP([N]^*\E)=\PP^r\times X $. Furthermore, since $f$ is dominated by $h$ this allows one to deduce the Kawaguchi--Silverman conjecture for $f$. We reduce the general case to this one. We now make this precise.

\begin{lemma}\label{lemma:nsquaredlemma}
Let $A$ be an abelian variety defined over $\overline{\QQ}$ and $f\colon A \rightarrow A$ a surjective endomorphism. For any integer $n$, there is a surjective endomorphism $f^\prime \colon A \rightarrow A$ such that $f\circ [n^2]=[n^2]\circ f^\prime$, where $[n^2]$ is the multiplication by $n^2$ map.
\end{lemma}
\begin{proof}
Write $f = t_c\circ \psi$ where $\psi$ is an isogeny and $t_c$ is translation by $c \in A$. Choose $c^\prime \in A(\Qbar)$ such that $[n^2] c^\prime = c$. Set $f^\prime = t_{c^\prime} \circ \psi$. We have 
\begin{equation*}
[n^2]\circ f^\prime = [n^2] \circ t_{c^\prime} \circ \psi = t_c\circ [n^2]\circ \psi = t_c \circ \psi \circ [n^2] = f \circ [n^2]   
\end{equation*}
as claimed. 
\end{proof}

\begin{proposition}\label{prop:abeliantorsion}
Let $A$ be an abelian variety defined over $\Qbar$. When $\L_0, \L_1, \ldots, \L_r$ are torsion line bundles on $A$, the Kawaguchi--Silverman conjecture is true for $\PP\bigl(\bigoplus_{i=0}^r \L_i\bigr)$.
\end{proposition}
\begin{proof}
Let $N=\textnormal{lcm}(\textnormal{Ord}(\L_i))$. Consider the fibre product $\PP (\E)\times_{A} A$ given by the diagram
\begin{equation*}
\begin{tikzcd}
\PP(\E)\times_{A} A \arrow[d, "\eta" left] \arrow[r, "\pi^\prime"] & A \arrow[d, "{[N^2]}"] \\
\PP(\E) \arrow[r,"\pi"] & A
\end{tikzcd}
\end{equation*}
where $[N^2]$ denotes the multiplication by $N^2$ map on $A$. We obtain the isomorphism 
\begin{equation*}
    \PP (\E)\times_{A} A\cong \PP ([N^2]^* \E)\cong A\times \PP^r.
\end{equation*}
Suppose that we have a surjective morphism $f\colon \PP(\E)\rightarrow \PP(\E)$. After replacing $f$ with some iterate we may assume that we have a diagram
\begin{equation*}
\begin{tikzcd}
\PP(\E) \arrow[d, "\pi" left] \arrow[r, "f"] & \PP(\E) \arrow[d, "\pi"] \\
A \arrow[r,"g"] & A
\end{tikzcd}
\end{equation*}
with $g$ surjective. Applying Lemma~\ref{lemma:nsquaredlemma} to $g$ and $N$ we have a morphism $g^\prime\colon A\rightarrow A$ such that $[N^2]\circ g^\prime=g\circ [N^2]$. We have
\begin{align*}
[N^2]\circ g^\prime\circ \pi^\prime&=g\circ [N^2]\circ \pi^\prime\\&=g\circ \pi \circ \eta\\ &=\pi\circ f\circ\eta.
\end{align*}
Therefore we have a unique surjective morphism $\widehat{f}\colon \PP([N^2]^*\E)\rightarrow \PP([N^2]^*\E)$ making the following diagram commute
\begin{equation*}
\begin{tikzcd}
\PP({[N^2]^*\E}) \arrow[d, "\eta" left] \arrow[r, "\pi^\prime"] \arrow[rd, "\widehat{f}"] & A \arrow[rd, "g^\prime"] & ~ \\
\PP(\E) \arrow[rd, "f"] & \PP({[N^2]^*\E}) \arrow[d, "\eta" left] \arrow[r, "\pi^\prime"] & A \arrow[d, "{[N^2]}"] \\
~ & \PP(\E) \arrow[r,"\pi"] & A.
\end{tikzcd}
\end{equation*}
As $\PP([N^2]^*\E)\cong A\times \PP^r$ the Kawaguchi--Silverman conjecture is known for $\widehat{f}$ by \cite[Theorem~1.3]{sano2020products}. By \cite[Lemma~3.2]{matsuzawasanoshibata2018surfaces} the Kawaguchi--Silverman conjecture holds for $f$.
\end{proof}

We wish to verify the Kawaguchi--Silverman conjecture for $\PP(\bigoplus_{i=0}^r\L_i)$ where $\L_i$ are algebraically trivial line bundles on a non-abelian variety $X$. Since each $\L_i$ is the pull-back of a line bundle on the Albanese variety, we can use Proposition~\ref{prop:abeliantorsion}. We prove a more general statement about pullbacks of projective bundles.

\begin{proposition}\label{prop:pullbackReduction}
Let $\psi\colon X\rightarrow Y$ be a surjective endomorphism of projective varieties defined over $\overline{\QQ}$. Let $\E$ be vector bundle on $Y$. Suppose that we have a diagram
\begin{equation}\label{eq:bigpullback}
\begin{tikzcd}
\PP(\psi^* \E) \arrow [r, "f" above] \arrow[d, "\pi" left] & \PP(\psi^*\E) \arrow[d, "\pi" right] \\
X \arrow[r, "g" above] \arrow[d, "\psi" left] & X \arrow[d, "\psi" right] \\
Y \arrow[r, "h" above] & Y.
\end{tikzcd}
\end{equation}
Let $\eta\colon \PP(\psi^*\E)=\PP(\E)\times_YX\rightarrow \PP(\E), \pi \colon \PP(\psi^* \E) \rightarrow X, \pi' \colon \PP(\E) \rightarrow Y$ denote the canonical projections. Assume that there is a numerically trivial line bundle $\B$ on $Y$ such that $f^*\O_{\PP(\psi^*\E)}(1)\cong \O_{\PP(\psi^*\E)}(d)\otimes \pi^*\psi^*\B$ and
\[h^0(\pi^*g^*\psi^*\E^\vee\otimes \O_{\PP(\psi^*\E)}(d)\otimes \pi^*\psi^* \B)=h^0(\pi^{\prime *}h^*\E^\vee\otimes\O_{\PP (\E)}(d)\otimes \pi^{\prime *} \B).\]
There exists a surjective morphism $\widehat{f}\colon \PP(\E)\rightarrow \PP(\E)$ making the following diagram commute.
\begin{equation}
\label{eq:smallpullback}
\begin{tikzcd}
\PP(\psi^* \E) \arrow[r, "f" above] \arrow[d, "\eta" left] & \PP(\psi^* \E) \arrow[d, "\eta" right] \\
\PP(\E) \arrow[r, "\widehat{f}"] & \PP(\E).
\end{tikzcd}
\end{equation}
Moreover, if $\E$ is a nef vector bundle then $\lambda_1(f\vert_\pi)=\lambda_1(\widehat{f}\vert_{\pi^\prime})$.
\end{proposition}
\begin{proof}
The morphism $f$ corresponds to a surjection of sheaves \[\beta\colon \pi^*g^*\psi^*\E\twoheadrightarrow \O_{\PP(\psi^*\E)}(d)\otimes \pi^*\psi^* B.\] The morphism is a global section of \[\Hom(\pi^*g^*\psi^*\E,\O_{\PP(\psi^*\E)}(d)\otimes \pi^*\psi^* \B)\cong \pi^*g^*\psi^*\E^\vee\otimes \O_{\PP(\psi^*\E)}(d)\otimes \pi^*\psi^* \B. \]
We have $\pi^*g^*\psi^*=\pi^*\psi^*h^*=\eta^*\pi^{\prime *}h^*$ and $\O_{\PP(\psi^*\E)}(1)=\eta^*\O_{\PP \E}(1)$. It follows that $\beta$ is a global section of
\begin{align*}
\eta^*\pi^{\prime *}h^*\E^\vee\otimes \eta^*\O_{\PP (\E)}(d)\otimes 
\eta^*\pi^{\prime *} \B &= \eta^*(\pi^{\prime *}h^*\E^\vee\otimes\O_{\PP (\E)}(d)\otimes \pi^{\prime *} \B)\\&=\eta^*(\Hom(\pi^{\prime *}h^*\E,\O_{\PP (\E)}(d)\otimes \pi^{\prime *} \B).
\end{align*}
Since $\eta$ is the base change of a surjective map, it is surjective. Therefore $\eta^*$ is injective on global sections as $(\eta^*s)(x)=s(\eta(x))$. Furthermore, by the same argument $\eta^*$ preserves non-vanishing sections. By assumption, the dimensions of $\Hom(\pi^{\prime *}h^*\E,\O_{\PP E}(d)\otimes \pi^{\prime *} \B)$ and $\eta^*\Hom(\pi^{\prime *}h^*\E,\O_{\PP E}(d)\otimes \pi^{\prime *} \B)$ are equal, therefore $\eta^*$ is an isomorphism and gives an identification between the non-vanishing sections of $\Hom(\pi^{\prime *}h^*\E,\O_{\PP E}(d)\otimes \pi^{\prime *} \B)$ and $\eta^*\Hom(\pi^{\prime *}h^*\E,\O_{\PP E}(d)\otimes \pi^{\prime *} \B)$. Therefore, every morphism $f\colon \PP(\psi^*\E)\rightarrow \PP(\psi^* \E)$ satisfying our assumptions is induced by some morphism $\widehat{f}\colon \PP(\E)\rightarrow \PP(\E)$ as claimed. The statement of relative dynamical degrees follows from the observation that since 
\begin{equation*}
f^*\O_{\PP(\psi^*\E)}(1)=\O_{\PP(\psi^*\E)}(d)\otimes\pi^*\psi^*\B,
\end{equation*}
we have that $\lambda_1(f\vert_\pi)=d$ when $\E$ is nef, and consequently so is $\psi^*\E$. On the the other hand, we also have that $\widehat{f}^*\O_{\PP(\E)}(1)=\O_{\PP(\E)}(d)\otimes\pi^{\prime *}\B$.
\end{proof}

\begin{corollary}\label{cor:pullbackRed}
Using the notation and assumptions of Proposition~\ref{prop:pullbackReduction}, if $\E$ is nef and the Kawaguchi--Silverman conjecture is true for $\widehat{f}$ and $g$, then the Kawaguchi--Silverman conjecture is true for $f$.
\end{corollary}
\begin{proof}
Applying the product formula for dynamical degrees to the top square of the diagram \eqref{eq:bigpullback} gives that $\lambda_1(f)=\max\{\lambda_1(f\vert_\pi),\lambda_1(g)\}$. When $\lambda_1(g)=\lambda_1(f)$, the Kawaguchi--Silverman conjecture is true for $f$ by Corollary~\ref{cor:standard1}. Hence we assume that $\lambda_1(f)=\lambda_1(f\vert_\pi)$. Proposition~\ref{prop:pullbackReduction} yields the diagram 
\begin{equation*}
\begin{tikzcd}
\PP(\E) \arrow[r, "\widehat{f}" above] \arrow[d, "\pi^\prime" left] & \PP(\E) \arrow[d, "\pi^\prime" right] \\
Y \arrow[r, "h"] & Y.
\end{tikzcd}
\end{equation*}
Applying the product formula gives that\[\lambda_1(\widehat{f})=\max\{\lambda_1(h),\lambda_1(\widehat{f}\vert_{\pi^\prime})\}.\] The bottom square of diagram \eqref{eq:bigpullback} gives that $\lambda_1(h)\leq \lambda_1(g)$ by direct computation. Therefore \[\lambda_1(\widehat{f})=\lambda_1(\widehat{f}\vert_{\pi^\prime})=\lambda_1(f\vert_\pi)=\lambda_1(f).\]
Consider the square \eqref{eq:smallpullback}. By our assumptions $\eta$ is surjective and the Kawaguchi--Silverman conjecture is true for $\widehat{f}$. Therefore by \cite[Lemma~3.2]{matsuzawasanoshibata2018surfaces} the Kawaguchi--Silverman conjecture is true for $f$.  
\end{proof}
The prototypical application of Corollary~\ref{cor:pullbackRed} is when $Y$ is the Albanese variety and $\psi$ the Albanese morphism. 
\begin{corollary}
\label{cor:mainThmAtorsion}
Let $X$ be a smooth projective variety over $\Qbar$ such that its Mori cone is generated by finitely many numerical classes of curves and such that the Kawaguchi--Silverman conjecture is true for all surjective endomorphisms of $X$. Fix $\L_0, \L_1,\ldots, \L_r$ to be torsion line bundles on $X$ and set $\E = \bigoplus_{i=0}^r \L_i$. For any surjective endomorphism of $\PP (\E)$, the Kawaguchi--Silverman conjecture holds.
\end{corollary}
\begin{proof}
Let $a\colon X\rightarrow \textnormal{Alb}(X)$ be the Albanese morphism. We have that $\L_i=a^*\W_i$ for some algebraically trivial line bundle on $\textnormal{Alb}(X)$. By Proposition~\ref{prop:abeliantorsion} we have that the Kawaguchi--Silverman conjecture is true for $\PP(\bigoplus_{i=0}^r\W_i)$. Let $f \colon \PP(\bigoplus_{i=0}^r\L_i)\rightarrow \PP(\bigoplus_{i=0}^r\L_i)$ be a surjective endomorphism. After iterating $f$ we may assume we have a diagram as in \eqref{eq:bigpullback} with $Y=\textnormal{Alb}(X)$ and $\psi=a\colon X\rightarrow \textnormal{Alb}(X)$. Write $\E=\bigoplus_{i=0}^r a^* \W_i$. By Proposition~\ref{prop:fibreProp} we have that $f^*\O_{\PP(a^*\E)}(1)\equiv_{\textnormal{num}} \O_{\PP(a^*\E)}(d)$ for some $d$. Therefore we obtain that the pullback is $f^*\O_{\PP(a^*\E)}(1)=\O_{\PP(a^*\E)}(d)\otimes a^*\B$ for some algebraically trivial line bundle $\B$ on $\textnormal{Alb}(X)$. We must now verify the second condition in the assumption of Proposition~\ref{prop:pullbackReduction}. We have that
\begin{align*}
H^0(\PP(a^*\E),\pi^*g^*\psi^*\E^\vee\otimes \O_{\PP(\psi^* E)}(d)
\otimes \pi^*a^*\B)&=H^0(X,g^*\psi^*\E^\vee\otimes\sym^d(a^*\E)\otimes a^*\B)\\ &=H^0(X,a^*(h^*\E^\vee\otimes \sym^d(\E)\otimes \B)).
\end{align*}
Now we have that 
\begin{align*}
&H^0(\textnormal{Alb}(X),h^*\E^\vee\otimes \sym^d(\E)\otimes \B)\\
&=\bigoplus_{i=0}^r\bigoplus_{d_0+\cdots d_r=d}H^0(\operatorname{Alb}(X),h^*\W_i^{-1}\otimes \B\otimes \W_0^{\otimes d_0}\otimes\cdots \otimes \W_r^{\otimes d^r}).
\end{align*}
As $\B$ and each of the $\W_i$ are algebraically trivial we have that they are either trivial, or have no nonzero global sections. Therefore $h^0(\operatorname{Alb}(X),h^*\E^\vee\otimes \sym^d(\E)\otimes \B)$ is the number of pairs $(i,d_0,\ldots ,d_r)$ such that $h^*\W_i^{-1}\otimes \B\otimes \W_0^{\otimes d_0}\otimes\cdots \otimes \W_r^{\otimes d^r}=\O_{\operatorname{Alb}(X)}$. On the other hand, 
\begin{align*}
&H^0(X,a^*(h^*\E^\vee\otimes \sym^d(\E)\otimes \B))\\
&=\bigoplus_{i=0}^r\bigoplus_{d_0+\cdots d_r=d}H^0(X,a^*(h^*\W_i^{-1}\otimes \B\otimes \W_0^{\otimes d_0}\otimes\cdots \otimes \W_r^{\otimes d^r})).
\end{align*}
Since $a$ is surjective, it is injective on global sections and we see that 
\begin{align*}
\phantom{\iff} &H^0(X,a^*(h^*\W_i^{-1}\otimes \B\otimes \W_0^{\otimes d_0}\otimes\cdots \otimes \W_r^{\otimes d^r}))\neq 0\\ \iff &H^0(\operatorname{Alb}(X),h^*\W_i^{-1}\otimes \B\otimes \W_0^{\otimes d_0}\otimes\cdots \otimes \W_r^{\otimes d^r})\neq 0\\ \iff &h^*\W_i^{-1}\otimes \B\otimes \W_0^{\otimes d_0}\otimes\cdots \otimes \W_r^{\otimes d^r}=\O_{\operatorname{Alb}}(X)\end{align*} 
It follows that the dimensions are the same so that we may apply Proposition~\ref{prop:pullbackReduction}. Applying Corollary~\ref{cor:pullbackRed} gives the claim. 
\end{proof}

With these two results in hand, we are able to fully resolve the \hyperref[conj:KSC]{KS}~Conjecture for the projectivization of any direct sum of line bundles on an elliptic curve.

\begin{proof}[Proof of Corollary~\ref{intro:cor2}]
When $\E$ is a direct sum of torsion line bundles, this follows from Corollary~\ref{cor:mainThmAtorsion}. Otherwise, it follows from Theorem~\ref{thm:mainThmAnontorsion} and Corollary~\ref{cor:standard1}.
\end{proof}

\section{Examples of Endomorphisms}
\label{sec:examples}
When $\E$ does not split as a direct sum of line bundles, Theorem~\ref{thm:nonsplitnoendos} shows that there are no non-trivial endomorphisms of $\PP(\E)$. However, if $\E = \bigoplus_{i=1}^n \L_i$ for some degree zero line bundles $\L_i$, then we can explicitly write down some non-trivial endomorphisms. Let us restrict to the case of $X = C$ being an elliptic curve.

\begin{example}
Suppose $\E = \O_C \oplus \L_1 \oplus \L_2$ where $\L_1$ and $\L_2$ are both torsion of degree $k$. Let $g$ be the endomorphism on $C$ given by multiplication by $k$ and $\{U,V\}$ an open cover of $C$ trivializing $\E$. We can set $\B = \O_C$, and
\begin{align*}
    F_0 &= t_0^k, & F_1 &= t_1^k, & F_2 &= t_2^k, \\
    G_0 &= t_0^k, & G_1 &= t_1^k, & G_2 &= t_2^k.
\end{align*}
We check that these satisfy the conditions of Lemma~\ref{lemma:explicitcompconds}. Indeed, since $\L_1$ and $\L_2$ are torsion of order $k$, we have that $g^* \L_1 = g^* \L_2 = \O_C$. Hence the compatibility conditions are just that $\bigl(\Sym^d M\bigr)(F_i(t_0,t_1,t_2)) = F_i(t_0, \alpha_1 t_1, \alpha_2 t_2) = G_i$, where $\alpha_1$ and $\alpha_2$ are transition functions for $\L_1$ and $\L_2$ respectively. Since these bundles are both torsion of order $k$, we could choose $\alpha_1$ and $\alpha_2$ such that $\alpha_1^k = \alpha_2^k = 1$. It follows that the compatibility conditions are satisfied, and so this choice of $F_i$ and $G_i$ defines a surjective endomorphism of $\PP(\E)$ with degree $k$ on the fibres. In particular, this is a non-trivial endomorphism.
\end{example}

When the characteristic of the field is nonzero, we can even find non-trivial surjective endomorphisms of Atiyah bundles.

\begin{example}
Suppose we are working over $\FF_5$. Set $\lambda = 2$, so that the curve $C$ is given by equation
\begin{equation*}
    Z Y^2 = X(X - Z)(X - 2Z).
\end{equation*}
For convenience, we trivialize to the open set $U$ and write $x = \frac{X}{Z}, y = \frac{Y}{Z}$ so that $C$ is defined by $f \coloneqq y^2 - x(x - 1)(x - 2) = 0$. Set $\E = \F_2$ and
\begin{align*}
    F_0 &= -2t_0^5, & F_1 &= t_1^5 + (y - xy) t_0^5, \\
    G_0 &= -2t_0^5, & G_1 &= t_1^5 + (x^2 y^{-5} + 2x^2 y^{-3} - xy^{-1} + y^{-1}) t_0^5.
\end{align*}
Note that $y - xy \in \O_C(U) = \FF_5[x,y]/\langle f \rangle$ and $x^2 y^{-5} + 2x^2 y^{-3} - xy^{-1} + y^{-1} \in \O_C(V) = \FF_5[xy^{-1}, y^{-1}]/\langle y^{-3} f \rangle$. One can explicitly compute a reduction of the element $\omega^5 = x^{10} y^{-5}$ via $f$ to see that
\begin{equation*}
\omega^5 \equiv -y + xy + x^2 y^{-5} + 2x^2 y^{-3} - xy^{-1} + y^{-1} - 2\omega.
\end{equation*}
The first compatibility condition in Lemma~\ref{lemma:explicitcompconds} requires $\bigl( \Sym^d M \bigr) \bigl( F_0(t_0, t_1) \bigr) = F_0(t_0, t_1 + \omega t_0) = G_0$, and the second requires that
\begin{align*}
    \bigl(\Sym^d M\bigr)(F_1) &= (t_1 + \omega t_0)^5 + (y - xy) t_0^5 \\
    &= t_1^5 + (-y + xy + x^2 y^{-5} + 2x^2 y^{-3} - xy^{-1} + y^{-1} - 2\omega) t_0^5 + (y - xy) 0_1^5 \\
    &= t_1^5 + (x^2 y^{-5} + 2x^2 y^{-3} - xy^{-1} + y^{-1} - 2\omega) t_0^5 \\
    &= G_0 - 2\omega t_0^5 \\
    &= G_0 + \omega G_1
\end{align*}
as desired. Hence this defines a surjective endomorphism of $\PP(\E)$. In particular, this is a non-trivial endomorphism with degree 5 on the fibres. In fact, we could also construct an endomorphism of degree being any multiple of $5$. Similarly, this construction would work over any prime characteristic which is at least 5.
\end{example}

\section*{Acknowledgements}

The authors thank Yohsuke Matsuzawa, Matthew Satriano, Gregory G. Smith for helpful feedback, and Mike Roth for insightful comments, especially regarding the arguments in Section~\ref{sec:nonsplit}. We are very grateful to anonymous reviewers for their suggestions improving the quality of exposition. The second author was partially funded by the Natural Sciences and Engineering Research Council (NSERC).

\bibliographystyle{plain}
\bibliography{refs}

\Addresses

\end{document}